\documentclass[11pt,twoside]{article}
\usepackage[T1]{fontenc}
\usepackage{textcomp}
\usepackage{mathrsfs}
\usepackage{amsmath}
\usepackage{amssymb}
\usepackage{fancyhdr}
\usepackage{latexsym}
\usepackage{bbding}
\usepackage{mathrsfs}
\usepackage{wasysym}
\usepackage{cite}
\usepackage{multicol,graphics}
\usepackage{xcolor}

\RequirePackage{times}

\def\XXint#1#2#3{{\setbox0=\hbox{$#1{#2#3}{\int}$ }
		\vcenter{\hbox{$#2#3$ }}\kern-.6\wd0}}

\newtheorem{theorem}{Theorem}[section]
\newtheorem{lemma}[theorem]{Lemma}

\newtheorem{definition}[theorem]{Definition}

\newtheorem{proposition}[theorem]{Proposition}
\numberwithin{equation}{section}
\newenvironment{proof}[1][Proof]{\noindent\textbf{#1.} }{\hfill $\Box$}

\allowdisplaybreaks[4]
\makeatletter
\begin{document}
	
\title{Global well-posedness of the fractional dissipative system in the framework of variable Fourier--Besov spaces\footnote{E-mail: gvh@westlake.edu.cn, jihzhao@163.com.}}
\author{Gast\'on Vergara-Hermosilla$^{\text{1}}$, Jihong Zhao$^{\text{2}}$\\
[0.2cm] {\small $^{\text{1}}$Institute for Theoretical Sciences, Westlake University, Hangzhou, Zhejiang 310024, China.}\\
[0.2cm] {\small $^{\text{2}}$School of Mathematics and Information Science, Baoji University of Arts and Sciences,}\\
[0.2cm] {\small  Baoji, Shaanxi 721013,  China}}

\date{\today}
\maketitle
\begin{abstract}
In this paper, we are concerned with the well-posed issues of the fractional dissipative system in the framework of the Fourier--Besov spaces with variable regularity and integrability indices. By fully using some basic properties of these variable function spaces, we establish the linear estimates in variable Fourier--Besov spaces for the fractional heat equation. Such estimates are fundamental for solving certain dissipative PDE's of fractional type. As an applications, we prove global well-posedness in variable Fourier--Besov spaces for the 3D generalized incompressible  Navier--Stokes equations and the 3D fractional Keller--Segel system.
\end{abstract}
\smallbreak

\smallskip


\section{Introduction}

The purpose of this paper is to develop a general method to establish the global well-posedness of the Cauchy problem of the generalized incompressible Navier--Stokes equations and the fractional Keller--Segel system in the framework of the variable Fourier--Besov spaces through the abstract Banach fixed point method.  Both of the two equations can be formulated as the following fractional dissipative form:
\begin{equation}\label{eq1.1}
\begin{cases}
\partial_{t}u+\Lambda^{\alpha} u=F(u), \ \ \ &t\in \mathbb{R}^{+}, x\in \mathbb{R}^{n},\\
u(0,x)=u_{0}(x), \ \ \  &x\in \mathbb{R}^{n},
\end{cases}
\end{equation}
where $0<\alpha\leq2$ is a real number and $F(u)$ is the nonlinear term.  The fractional power of Laplacian $\Lambda^{\alpha}=(-\Delta)^{\frac{\alpha}{2}}$ is defined by
\begin{equation*}
  \Lambda^{\alpha}u(x):
    =\sigma_{n,\alpha}P.V.\int_{\mathbb{R}^{n}}\frac{u(x)-u(y)}{|x-y|^{n+\alpha}}dy
\end{equation*}
and $\sigma_{n,\alpha}$ is a normalization constant.
\smallbreak

By the classical fixed point method, the well-posedness of the incompressible Navier--Stokes equations and the Keller--Segel system in the Fourier--Besov spaces $F\dot{B}^{s}_{p,r}(\mathbb{R}^{n})$ with constant regularity and integrability indices have been obtained. In \cite{I11},  Iwabuchi first introduced the Fourier--Besov spaces $F\dot{B}^{s}_{p,r}(\mathbb{R}^{n})$ to study the global well-posed and ill-posed issues of the parabolic-elliptic Keller--Segel system. Later on,  Iwabuchi--Takada \cite{IT14} used the critical Fourier--Besov spaces $F\dot{B}^{s}_{p,r}$-spaces to establish the global well-posedness (uniformly with respect to
the angular velocity $\Omega$) for the Navier--Stokes--Coriolis system. Konieczny--Yoneda \cite{KY11} showed global well--posedness and
asymptotic stability of small solutions for the 3D Navier--Stokes equations in the critical Fourier--Besov spaces. We refer to \cite{ZZ18,ZZ19} to more well-posed results in Fourier--Besov spaces.
\smallbreak

On the other hand, spaces of variable integrability, also known as variable Lebesgue spaces, are generalization of the classical Lebesgue spaces, replacing the constant $p$ with a variable exponent function $p(\cdot)$.  The origin of the variable Lebesgue spaces $L^{p(\cdot)}(\mathbb{R}^{n})$ can be traced back to Orlicz \cite{O31} in 1931, and then studied two decades later by Nakano \cite{N50, N51}, but the modern development started with the foundational paper of Kov\'{a}\v{c}ik--R\'{a}kosn\'{i}k \cite{KR91}, as well as the papers of Cruz-Uribe \cite{C03} and Diening \cite{D04}. Notice that not all properties of the usual Lebesgue spaces $L^{p}(\mathbb{R}^{n})$ can be generalized to the variable Lebesgue spaces $L^{p(\cdot)}(\mathbb{R}^{n})$. For example,  the variable Lebesgue spaces are not translation invariant, thus the convolution of two functions $f$ and $g$ is not well-adapted, and the Young's inequality are not valid anymore (cf. \cite{CF13}, Section 5.3). In consequence, it is difficult to extend some classical well-posed results results of evolution equations to the variable function spaces. Motivated by some well-posed results in the variable function spaces obtained recently in \cite{R18, RA19, V241, VZ241, VZ242}, we combine the Fourier splitting method and the product estimates in Fourier--Besov spaces to establish the global well-posedness of the evolution equations of fractional type in the variable Fourier--Besov spaces.
\smallbreak

To illustrate the main idea of our approach, we first revisit the Cauchy problem of the fractional heat equation:
\begin{equation}\label{eq1.2}
\begin{cases}
\partial_{t}u+\Lambda^{\alpha} u=f(t,x), \ \ \ &t\in \mathbb{R}^{+}, x\in \mathbb{R}^{n},\\
u(0,x)=u_{0}(x), \ \ \  &x\in \mathbb{R}^{n}.
\end{cases}
\end{equation}
We obtain the following solvability result in the variable Fourier--Besov spaces.
\begin{theorem}\label{th1.1}
Let $0<T\leq \infty$, $1\leq \rho,r\leq\infty$ and $s(\cdot), p(\cdot)\in \mathcal{P}_{0}^{\log}(\mathbb{R}^{n})$. Assume that $u_{0}\in F\dot{B}^{s(\cdot)}_{p(\cdot),r}(\mathbb{R}^{n})$, $f\in\mathcal{L}^{\rho}(0,T; F\dot{B}^{s(\cdot)-\alpha+\frac{\alpha}{\rho}}_{p(\cdot),r}(\mathbb{R}^{n}))$. Then the problem \eqref{eq1.2} admits a unique solution
$$
  u\in \mathcal{L}^{\infty}(0,T; F\dot{B}^{s(\cdot)}_{p(\cdot),r}(\mathbb{R}^{n}))\cap \mathcal{L}^{\rho}(0,T; F\dot{B}^{s(\cdot)+\frac{\alpha}{\rho}}_{p(\cdot),r}(\mathbb{R}^{n})).
$$
Moreover, for any $\rho_{1}\in[\rho,\infty]$, there exists a constant $C$ depending only on $n$ such that
\begin{equation}\label{eq1.3}
\|u\|_{\mathcal{L}^{\rho_{1}}_{T}(F\dot{B}^{s(\cdot)+\frac{\alpha}{\rho_{1}}}_{p(\cdot),r})}\leq C\left(\|u_{0}\|_{ F\dot{B}^{s(\cdot)}_{p(\cdot),r}}+\|f\|_{\mathcal{L}^{\rho}_{T}(F\dot{B}^{s(\cdot)-\alpha+\frac{\alpha}{\rho}}_{p(\cdot),r})}\right).
\end{equation}
\end{theorem}
\begin{proof}
Since $u_{0}$ and $f$ are temperate distributions, the fractional heat equation \eqref{eq1.2} admits a unique solution $u$ in $\mathcal{S}'((0,T)\times \mathbb{R}^{n})$, which satisfies
\begin{equation*}
\widehat{u}(t,\xi)=e^{-t|\xi|^{\alpha}}\widehat{u}_{0}(\xi)+\int_{0}^{t}e^{-(t-\tau)|\xi|^{\alpha}}\widehat{f}(\tau,\xi)d\tau.
\end{equation*}
Let $\phi_{j}$ be a smooth radial  non-increasing  function such that it is supported in the ring $\{|\xi|\sim2^{j}\}$. Then there exists a constant $\kappa>0$ such that
\begin{align*}
  \|2^{js(\cdot)}\phi_{j}\widehat{u}(t,\xi)\|_{L^{\rho_{1}}_{T}(L^{p(\cdot)}_{\xi})}&\leq \|e^{-\kappa t2^{\alpha j}}2^{js(\cdot)}\phi_{j}\widehat{u}_{0}(\xi)\|_{L^{\rho_{1}}_{T}(L^{p(\cdot)}_{\xi})}\nonumber\\
  &+\left\|\int_{0}^{t}e^{-\kappa(t-\tau)2^{\alpha j}}2^{js(\cdot)}\phi_{j}\widehat{f}(\tau,\xi)d\tau\right\|_{L^{\rho_{1}}_{T}(L^{p(\cdot)}_{\xi})}.
\end{align*}
Applying Young's inequality with respect to the time variable, one gets
\begin{align}\label{eq1.4}
  \|2^{js(\cdot)}\phi_{j}\widehat{u}(t,\xi)\|_{L^{\rho_{1}}_{T}(L^{p(\cdot)}_{\xi})}&\leq \left(\frac{1-e^{-\kappa 2^{\alpha j}\rho_{1}T}}{\kappa2^{\alpha j}\rho_{1}}\right)^{\frac{1}{\rho_{1}}}\|2^{js(\cdot)}\phi_{j}\widehat{u}_{0}(\xi)\|_{L^{p(\cdot)}_{\xi}}\nonumber\\
  &+ \left(\frac{1-e^{-\kappa 2^{\alpha j}\rho_{2}T}}{\kappa2^{\alpha j}\rho_{2}}\right)^{\frac{1}{\rho_{2}}}\left\|2^{js(\cdot)}\phi_{j}\widehat{f}(\tau,\xi)\right\|_{L^{\rho}_{T}(L^{p(\cdot)}_{\xi})},
\end{align}
where $1+\frac{1}{\rho_{1}}=\frac{1}{\rho}+\frac{1}{\rho_{2}}$. Finally, taking $\ell^{r}$-norm to both sides of \eqref{eq1.4}, we conclude that
\begin{align*}
  \|u\|_{\mathcal{L}^{\rho_{1}}_{T}(F\dot{B}^{s(\cdot)+\frac{\alpha}{\rho_{1}}}_{p(\cdot),r})}&\leq \left[\sum_{j\in\mathbb{Z}}\left(\frac{1-e^{-\kappa 2^{\alpha j}\rho_{1}T}}{\kappa\rho_{1}}\right)^{\frac{r}{\rho_{1}}}\|2^{js(\cdot)}\phi_{j}\widehat{u}_{0}(\xi)\|_{L^{p(\cdot)}_{\xi}}^{r}\right]^{\frac{1}{r}}\nonumber\\
  &+ \left[\sum_{j\in\mathbb{Z}}\left(\frac{1-e^{-\kappa 2^{\alpha j}\rho_{2}T}}{\kappa \rho_{2}}\right)^{\frac{r}{\rho_{2}}}\left\|2^{j(s(\cdot)-\alpha+\frac{\alpha}{\rho})}\phi_{j}
  \widehat{f}(\tau,\xi)\right\|_{L^{\rho}_{T}(L^{p(\cdot)}_{\xi})}^{r}\right]^{\frac{1}{r}}\nonumber\\
  &\leq C\left(\|u_{0}\|_{ F\dot{B}^{s(\cdot)}_{p(\cdot),r}}+\|f\|_{\mathcal{L}^{\rho}_{T}(F\dot{B}^{s(\cdot)-\alpha+\frac{\alpha}{\rho}}_{p(\cdot),r})}\right).
\end{align*}
Thus we get $u\in \mathcal{L}^{\infty}(0,T; F\dot{B}^{s(\cdot)}_{p(\cdot),r}(\mathbb{R}^{n}))\cap \mathcal{L}^{\rho}(0,T; F\dot{B}^{s(\cdot)+\frac{\alpha}{\rho}}_{p(\cdot),r}(\mathbb{R}^{n}))$ and the inequality \eqref{eq1.3}.
\end{proof}
 \smallbreak

Now we state our main results concerning about the global well-posedness of the 3D generalized Navier--Stokes equations and the 3D fractional  Keller--Segel system in the Fourier--Besov spaces with variable regularity and integrability indices.
\subsection{The generalized incompressible Navier--Stokes equations}
The Cauchy problem of the generalized incompressible Navier--Stokes equations reads
\begin{equation} \label{eq3.1}
\left\{
\begin{aligned}
& \partial_t u +\Lambda^{\alpha} u + u\cdot\nabla u +\nabla \pi =f,  \\
& {\rm div} u=0,\\
& u(0,x) =u_0(x),
\end{aligned}
\right.
\end{equation}
where  the unknown functions $u$ and $\pi$ stand for the velocity and the pressure of the fluid, respectively, $u_{0}$ is a given initial velocity with compatibility condition ${\rm div}u_{0}=0$ and $f$ is a given external force. Notice that the divergence free condition determines the pressure $\pi$ through the relation
\begin{equation*}\label{eq3.2}
 -\Delta \pi=\sum_{j,k=1}^{n}\partial_{x_{j}}\partial_{x_{k}}(u^{j}u^{k}),
\end{equation*}
where $u^{j}$ is the $j$-th component of the velocity field $u$.
\smallbreak

Equations \eqref{eq3.1} is invariant under a particular change of time and space scaling. More precisely, if $u$ solves \eqref{eq3.1} with initial data $u_{0}$ and $f$, so does for any $\lambda>0$ the vector field $u_{\lambda}:(t,x)\mapsto\lambda^{\alpha-1} u(\lambda^{\alpha}t,\lambda x)$ with initial data $u_{0, \lambda}$ and $f_{\lambda}$, where
\begin{equation*}
 u_{0, \lambda}: x\mapsto\lambda^{\alpha-1} u_{0}(\lambda x) \ \ \ \text{and}\ \ \  f_{\lambda}:(t,x)\mapsto\lambda^{2\alpha-1} f(\lambda^{\alpha}t,\lambda x).
\end{equation*}
This scaling invariance is particularly significant for the equations \eqref{eq3.1} and leads to the following definition. A function space is called critical for the equations \eqref{eq3.1} if it is invariant under the scaling
\begin{equation}\label{eq3.2}
 u_{0,\lambda}(x)=\lambda^{\alpha-1} u_{0}(\lambda x).
\end{equation}

When $\alpha=2$, the equations \eqref{eq3.1} becomes the following classical incompressible Navier--Stokes equations:
\begin{equation} \label{eq3.3}
\left\{
\begin{aligned}
& \partial_t u -\Delta u + u\cdot\nabla u +\nabla \pi =f,  \\
& {\rm div} u=0,\\
& u(0,x) =u_0(x).
\end{aligned}
\right.
\end{equation}
Clearly,  under the scaling \eqref{eq3.2} with $\alpha=2$, the space $ \dot{H}^{\frac{n}{2}-1}(\mathbb{R}^{n})$ is a critical space for the equations \eqref{eq3.3}.  The global well-posedness of \eqref{eq3.3} with small initial data in $\dot{H}^{\frac{n}{2}-1}(\mathbb{R}^{n})$ has been proved by  Fujita--Kato \cite{FK64}. Since then, a number of works have been devoted to establishing similar well-posed results for larger class of initial data. Here we list some of them as follows.
 \begin{itemize}
 \item Kato \cite{K84} proved the global well-posedness for small data and local well-posedness for large data in the critical Lebesgue space $L^{n}(\mathbb{R}^{n})$, see also Giga \cite{G86}.

 \item  Barraza \cite{B96, B99} and Yamazaki \cite{Y00} established the global existence, regularity and stability of solutions in the critical Lorentz spaces  $L^{n,\infty}(\mathbb{R}^{n})$.

 \item  Cannone \cite{C97} and Planchon \cite{P96} proved the same results for small initial data in the  critical Besov spaces $\dot{B}^{-1+\frac{n}{p}}_{p,\infty}(\mathbb{R}^{n})$, see also Chemin \cite{C92}.

\item  Koch--Tataru \cite{KT01} proved the global well-posedness for small data in the space $BMO^{-1}$, which might be regarded as the largest space so that the Navier--Stokes equations \eqref{eq3.3} is known to be well-posed.

\item  Bourgain--Pavlovi\'{c} \cite{BP08} proved ill-posedness for the equations \eqref{eq3.3} in the largest scaling invariant space  $\dot{B}^{-1}_{\infty,\infty}(\mathbb{R}^{3})$, see Yoneda \cite{Y10} and Wang \cite{W15} for more ill-posedness results.

 \item  The literatures listed here are far from being complete; we refer the readers to see \cite{L16} for more expositions and references cited therein.
 \end{itemize}
\smallbreak

In the scenario where $1<\alpha<2$,  the generalized Navier--Stokes equations \eqref{eq3.1} was proposed in \cite{MGS13} to describe fluids with internal friction, and it also obtained from a stochastic Lagrangian particle approach by \cite{Z12}. Concerning about the well-posedness of the equations \eqref{eq3.1}, Lions \cite{L69} first proved global existence and
uniqueness of the classical solution when $\alpha\geq \frac{5}{2}$ in dimension $3$,  and the solution satisfies the following energy inequality:
\begin{equation*}\label{eq3.5}
 \int_{\mathbb{R}^{3}}|u|^{2}dx +\int_{0}^{T}|\Lambda^{\frac{\alpha}{2}}u|^{2}dxdt\leq  \int_{\mathbb{R}^{3}}|u|^{2}dx\ \ \ \text{for}\ \ \ 0<t\leq T.
\end{equation*}
We also list some well-posedness results of the equations \eqref{eq3.1} as follows.
 \begin{itemize}
 \item Wu \cite{W03} proved global regularity of solutions when $\alpha\geq 1+\frac{n}{2}$ in general dimension $n$.

 \item  Wu \cite{W05} proved the global well-posedness with small initial data $u_0$ in the critical Besov spaces $\dot{B}^{1-\alpha+\frac{n}{p}}_{p,q}(\mathbb{R}^{n})$ for $1<\alpha\leq2$, $2<p<\infty$ and $1\leq q \leq \infty$.

    \item Yu--Zhai \cite{YZ14} established the global existence and uniqueness of solutions  with small initial data in the critical Besov space $\dot{B}^{1-\alpha}_{\infty,\infty}(\mathbb{R}^{n})$ for $1<\alpha<2$.

     \item Li--Zhai \cite{LZ10} established the global well-posdness and regularity with small initial data in some critical $Q$-spaces for $1<\alpha< 2$.

    \item Huang--Wang \cite{HW13} proved the global well-posedness and Gevrey analyticity of solutions with small data  either in critical Besov spaces $\dot{B}^{\frac{n}{p}}_{p,1}(\mathbb{R}^{n})$ ($1\leq p<\infty$) or in $\dot{B}^{0}_{\infty,1}(\mathbb{R}^{n})$ for $\alpha=1$. Similar results also hold for $1<\alpha<2$.
 \end{itemize}
\smallbreak

Recently, some results on global well-posedness in the variable function spaces have been obtained for the 3D generalized Navier--Stokes equations \eqref{eq3.1}. In \cite{V241}. by overcoming the difficulties caused by the boundedness of the Riesz potential in a variable Lebesgue spaces,  the first author of this paper proved the global well-posedness of the equations \eqref{eq3.1} in the framework of variable Lebesgue spaces $L^{p(\cdot)}(\mathbb{R}^{3})\cap L^{\frac{3}{\alpha-1}}(\mathbb{R}^{3})$. Ru--Abidin \cite{RA19} established the global well-posedness of the equations \eqref{eq3.1} with small initial data in the variable Fourier--Besov spaces $F\dot{B}^{4-\alpha- \frac{3}{p(\cdot)}}_{p(\cdot),r}(\mathbb{R}^{3})$.  Similar results also hold for the classical Navier--Stokes equations \eqref{eq3.3}, see \cite{CV24, R18, VZ242}.
\smallbreak

Motivated by the above well-posed results in variable function spaces, the first purpose of this paper is to establish the global well-posedness of the 3D generalized Navier--Stokes equations \eqref{eq3.1} in the framework of the Fourier--Besov spaces with variable regularity and integrability indices. The main result is as follows.

\begin{theorem}\label{th3.1}
Let $n=3$,  $1<\alpha\leq 2$, $p(\cdot)\in \mathcal{P}^{\log}_{0}(\mathbb{R}^{3})$ with $2\leq p^{-}\leq p(\cdot) \leq p^{+}\leq \frac{6}{5-2\alpha}$, $1\leq \rho\leq+\infty$. For any $u_{0}\in F\dot{B}^{4-\alpha-\frac{3}{p(\cdot)}}_{p(\cdot),1}(\mathbb{R}^{3})$ and  $f\in\mathcal{L}^{1}(0,\infty; F\dot{B}^{4-\alpha-\frac{3}{p(\cdot)}}_{p(\cdot),1}(\mathbb{R}^{3}))\cap \mathcal{L}^{1}(0,\infty; F\dot{B}^{\frac{5}{2}-\alpha}_{2,1}(\mathbb{R}^{3}))$, there exists a positive constant $\varepsilon$ such that if the initial data $u_{0}$ and $f$ satisfy
\begin{equation*}
\|u_{0}\|_{F\dot{B}^{4-\alpha-\frac{3}{p(\cdot)}}_{p(\cdot),1}}+\|f\|_{\mathcal{L}^{1}_{t}(F\dot{B}^{4-\alpha-\frac{3}{p(\cdot)}}_{p(\cdot),1})}
+\|f\|_{\mathcal{L}^{1}_{t}(F\dot{B}^{\frac{5}{2}-\alpha}_{2,1})}\leq \varepsilon,
\end{equation*}
then the generalized Navier--Stokes equations \eqref{eq3.1} admits a unique global solution $u$ such that
\begin{equation*}
u\in \mathcal{L}^{\rho}(0,\infty; F\dot{B}^{4-\alpha-\frac{3}{p(\cdot)}+\frac{\alpha}{\rho}}_{p(\cdot),1}(\mathbb{R}^{3}))\cap \mathcal{L}^{1}(0,\infty; F\dot{B}^{\frac{5}{2}}_{2,1}(\mathbb{R}^{3}))\cap \mathcal{L}^{\infty}(0,\infty; F\dot{B}^{\frac{5}{2}-\alpha}_{2,1}(\mathbb{R}^{3})).
\end{equation*}
\end{theorem}

\noindent\textbf{Remark 1.1}  Let us remark that, the techniques we used to prove Theorem \ref{th3.1} is different from \cite{RA19}, we shall employ the Fourier--Besov spaces with constant exponents, rather than Besov type spaces used in \cite{RA19}, to be the auxiliary spaces to find global solutions of the equations \eqref{eq3.1}. In our view, choosing the Fourier--Besov spaces with constant exponents as auxiliary spaces seems more natural way to find solutions of PDEs in the framework of variable Fourier--Besov spaces.

\noindent\textbf{Remark 1.2} Theorem \ref{th3.1} can be regarded as a generalization of global well-posed results of \cite{VZ242} to the classical Navier--Stokes equations \eqref{eq3.1}.

\subsection{The fractional Keller--Segel system}

The Cauchy problem of the fractional Keller--Segel system writes
\begin{equation}\label{eq3.4}
\begin{cases}
  \partial_{t}u+\Lambda^{\alpha}u+\nabla\cdot(u\nabla \phi)=f,\\
    -\Delta \phi=u,\\
  u(0,x)=u_0(x),
\end{cases}
\end{equation}
where $u$ and $\phi$ are two unknown functions which stand for the cell density and the concentration of the chemical attractant, respectively, $u_{0}$ and $f$ are given initial data and an external force. Notice that the chemical concentration $\phi$ can be represented as the volume potential of $u$:
\begin{equation*}
\phi(t,x)=(-\Delta)^{-1}u(t,x)=
\begin{cases}
\frac{1}{n(n-2)\omega_{n}}\int_{\mathbb{R}^{n}}\frac{u(t,y)}{|x-y|^{n-2}}dy,
\quad n\geq 3,\\
-\frac{1}{2\pi}\int_{\mathbb{R}^{2}}u(t, y)\log|x-y|dy,
\quad n=2,
\end{cases}
\end{equation*}
where $\omega_{n}$ denotes the surface area of the unit sphere in
$\mathbb{R}^{n}$.
\smallbreak

System \eqref{eq3.1} is also invariant under the following scaling: if $u$ solves \eqref{eq3.4} with initial data $u_{0}$ and $f$, so does for any $\lambda>0$ the function $u_{\lambda}:(t,x)\mapsto\lambda^{\alpha} u(\lambda^{\alpha}t,\lambda x)$ with initial data $u_{0, \lambda}$ and $f_{\lambda}$, where
\begin{equation*}
 u_{0, \lambda}: x\mapsto\lambda^{\alpha} u_{0}(\lambda x) \ \ \ \text{and}\ \ \  f_{\lambda}:(t,x)\mapsto\lambda^{2\alpha} f(\lambda^{\alpha}t,\lambda x).
\end{equation*}
A function space is called critical for the system \eqref{eq3.4} if it is invariant under the scaling
\begin{equation}\label{eq3.5}
 u_{0,\lambda}(x)=\lambda^{\alpha} u_{0}(\lambda x).
\end{equation}

The system \eqref{eq3.4} was first proposed by Escudero in \cite{E06},  where it was used to describe the spatiotemporal distribution of a population density of random walkers undergoing L\'{e}vy flights.  When $\alpha=2$, the system \eqref{eq3.4} becomes to the well-known elliptic-parabolic Keller--Segel system of chemotaxis (cf. \cite{KS70}):
\begin{equation}\label{eq3.6}
\begin{cases}
  \partial_{t}u-\Delta u=-\nabla\cdot(u\nabla \phi),\\
  -\Delta \phi=u,\\
  u(0,x)=u_0(x).
\end{cases}
\end{equation}
 It is well-known that the system \eqref{eq3.6}  admits global solutions if the initial total mass $\int_{\mathbb{R}^{n}}u_{0}(x)dx$ ($n\ge2$) is small enough, and it may blow up in finite time for large initial data, we refer the readers to see \cite{H03, H04} for a comprehensive review of these aspects. Similar to the Navier--Stokes equations \eqref{eq3.3}, a number of works have been devoted to establishing global well-posedness of the system \eqref{eq3.6} with small initial data in the scale invariant spaces. Here we list some of them as follows.
\begin{itemize}
\item Corrias--Perthame--Zaag \cite{CPZ04} proved global existence of solutions with initial data $ u_{0}\in L^{1} (\mathbb{R}^{n})\cap L^{\frac{n}{2}}(\mathbb{R}^{n})$ and
$\| u_{0}\|_{L^{\frac{n}{2}}}$ being sufficiently small, see also Kozono--Sugiyama \cite{KS08}.

\item Biler et al. \cite{BCGK04} established global well-posedness with small initial data in critical pseudomeasure space $\mathcal{PM}^{n-2}(\mathbb{R}^{n})$.

 \item In dimension 2, Ogawa--Shimizu \cite{OS08, OS10} established global well-posedness with small initial data in critical Hardy space $\mathcal{H}^{1}(\mathbb{R}^{2})$ and Besov space $\dot{B}^{0}_{1,2}(\mathbb{R}^{2})$, respectively;

\item Karch \cite{K99} and Iwabuchi \cite{I11} proved the global well-posedness with small initial data in the critical Besov spaces $\dot{B}^{-2+\frac{n}{p}}_{p,q}(\mathbb{R}^{n})$ with $1\leq p<\infty$ and $1\leq q\leq\infty$, and the ill-posedness is also established in critical Besov type spaces $\dot{\mathcal{B}}^{-2}_{\infty,q}(\mathbb{R}^{n})$ with  $2<q\leq\infty$.

\item Iwabuchi--Nakamura \cite{IN13} showed global existence of solutions with small initial data in $BMO^{-2}$ through the Triebel--Lizorkin spaces $\dot{F}^{-2}_{\infty,2}(\mathbb{R}^{n})$.

\item For more well-posed and ill-posed results to more general chemotaxis models, we refer the readers to see \cite{LW23, LYZ23, NY20, NY22, XF22, Z24}.
\end{itemize}
 \smallbreak

For the general fractional diffusion case $1<\alpha<2$,  we know from \cite{E06, OY01} that the one dimensional system \eqref{eq3.4} possesses global solution not only in the case $\alpha=2$ but also in the case $1<\alpha<2$. However, a singularity appears in finite time when $0<\alpha<1$ if the
initial configuration of cells is sufficiently concentrated, see Bournaveas--Calvez \cite{BC10}. Moreover, when $n\ge2$, the solutions of \eqref{eq3.4} globally exist for small initial data and may blow up in finite time for large initial data, here we list some results concerning about the global existence or blow-up of solutions.
\begin{itemize}
\item For $1<\alpha<2$, Biler--Karch \cite{BK10}  established global well-posedness with small initial data in the critical Lebesgue space $L^{\frac{n}{\alpha}}(\mathbb{R}^{n})$, they also proved the finite time blowup of nonnegative solutions with initial data imposed on large mass or high concentration conditions.

 \item For $1<\alpha<2$, Biler--Wu \cite{BW09} established  global well-posedness with small initial data in  critical  Besov space $\dot{B}^{1-\alpha}_{2,q}(\mathbb{R}^{2})$.

\item For $1<\alpha\leq2$, Zhai \cite{Z10} proved global existence, uniqueness and stability of solutions in critical Besov space with general potential type nonlinear term.

\item For $1<\alpha \leq 2$, Wu--Zheng \cite{WZ11} proved global well-posedness with small initial data in critical Fourier--Herz space $\mathcal{\dot{B}}^{2-2\alpha}_{q}(\mathbb{R}^{n})$ for $1<\alpha\leq2$ and $2\leq q\leq \infty$, and they also proved  that system \eqref{eq3.4} is  ill-posedness in  $\mathcal{\dot{B}}^{-2}_{q}(\mathbb{R}^{n})$ and  $\dot{B}^{-2}_{\infty, q}(\mathbb{R}^{n})$ with $\alpha=2$ and $2<q\leq\infty$.

\item For $0<\alpha\leq 1$, Sugiyama--Yamamoto--Kato \cite{SYK15} proved local existence with large data and global existence with small data in critical Besov space $\dot{B}^{-\alpha+\frac{n}{p}}_{p,q}(\mathbb{R}^{n})$ for $n\geq 3$, $2\leq p<n$ and  $1\leq q<2$.

 \item For $1<\alpha\leq 2$, Zhao \cite{Z18} showed global existence and analyticity of solutions with small data in critical Besov space $\dot{B}^{-\alpha+\frac{n}{p}}_{p,q}(\mathbb{R}^{n})$ ($1\leq p<\infty$,  $1\leq q\leq \infty$) and $\dot{B}^{-\alpha}_{\infty,1}(\mathbb{R}^{n})$. For the limit case $\alpha=1$, the global existence and analyticity of solutions with small data in critical Besov space $\dot{B}^{-1+\frac{n}{p}}_{p,1}(\mathbb{R}^{n})$ ($1\leq p<\infty$) and $\dot{B}^{-1}_{\infty,1}(\mathbb{R}^{n})$ were also obtained.

 \item  Parts of these results were also extended for the other generalized chemotaxis models and the fractional power drift-diffusion system of bipolar type,  we refer the readers to see \cite{CLW19, DL17, MYZ08, SYK15, YKS14, Z21} for more details.
 \end{itemize}
\smallbreak

Recently, in \cite{VZ241, VZ242}, the authors in this paper established some well-posed results in variable function spaces for the system \eqref{eq3.4} and \eqref{eq3.6}. Indeed, in \cite{VZ241}, by carefully examining the algebraical structure, we first reduce the system \eqref{eq3.4} into the fractional nonlinear heat equation to overcome the difficulties caused by the boundedness of the Riesz potential in a variable Lebesgue spaces, then by mixing some structural properties of the variable Lebesgue spaces with the optimal decay estimates of the fractional heat kernel, we obtained global existence of solutions of the Keller--Segel system \eqref{eq3.4} in the variable Lebesgue spaces $L^{p(\cdot)}(\mathbb{R}^{n})\cap L^{\frac{n}{\alpha}}(\mathbb{R}^{n})$. In \cite{VZ242},  by fully using the linear estimates of the heat equation in variable Fourier--Besov spaces and some of the main properties of these functional spaces, we proved global well-posedness of the 3D Keller--Segel system \eqref{eq3.6} in the framework of variable Fourier--Besov spaces
$F\dot{B}^{1-\frac{3}{p(\cdot)}}_{p(\cdot),1}(\mathbb{R}^{3})$.
 \smallbreak

Based on the above well-posed results in variable function spaces, the second  goal of this paper is to prove the global well-posedness of the 3D fractional Keller--Segel system \eqref{eq3.4} in the framework of the Fourier--Besov spaces with variable regularity and integrability indices. The main result is as follows.
\begin{theorem}\label{th3.2}
Let $n=3$, $1<\alpha\leq 2$, $p(\cdot)\in \mathcal{P}^{\log}_{0}(\mathbb{R}^{3})$ with $2\leq p^{-}\leq p(\cdot) \leq p^{+}\leq \frac{6}{5-2\alpha}$, $1\leq \rho\leq+\infty$. For any $u_{0}\in F\dot{B}^{3-\alpha-\frac{3}{p(\cdot)}}_{p(\cdot),1}(\mathbb{R}^{3})$ and $f\in\mathcal{L}^{1}(0,\infty; F\dot{B}^{3-\alpha-\frac{3}{p(\cdot)}}_{p(\cdot),1}(\mathbb{R}^{3}))\cap \mathcal{L}^{1}(0,\infty; F\dot{B}^{\frac{3}{2}-\alpha}_{2,1}(\mathbb{R}^{3}))$, there exists a positive constant $\varepsilon$ such that if the initial data $u_{0}$ and $f$ satisfy
\begin{equation*}
\|u_{0}\|_{F\dot{B}^{3-\alpha-\frac{3}{p(\cdot)}}_{p(\cdot),1}}+\|f\|_{\mathcal{L}^{1}_{t}(F\dot{B}^{3-\alpha-\frac{3}{p(\cdot)}}_{p(\cdot),1})}
+\|f\|_{\mathcal{L}^{1}_{t}(F\dot{B}^{\frac{3}{2}-\alpha}_{2,1})}\leq \varepsilon,
\end{equation*}
then the fractional Keller--Segel system \eqref{eq3.4} admits a unique global solution $u$ such that
\begin{equation*}
u\in \mathcal{L}^{\rho}(0,\infty; F\dot{B}^{3-\alpha-\frac{3}{p(\cdot)}+\frac{\alpha}{\rho}}_{p(\cdot),1}(\mathbb{R}^{3}))\cap \mathcal{L}^{1}(0,\infty; F\dot{B}^{\frac{3}{2}}_{2,1}(\mathbb{R}^{3}))\cap \mathcal{L}^{\infty}(0,\infty; F\dot{B}^{\frac{3}{2}-\alpha}_{2,1}(\mathbb{R}^{3})).
\end{equation*}
\end{theorem}

\noindent\textbf{Remark 1.3} Theorem \ref{th3.2} can be regarded as a generalization of global well-posed results of \cite{VZ242} to the fractional Keller--Segel system \eqref{eq3.4}.
 \smallbreak

The structure of this paper is arranged as follows. In Section 2, we introduce some conventions and notations, and state the definition and some basic results of the homogeneous (variable) Fourier--Besov spaces. In Section 3, we present the proof of Theorem \ref{th3.1}, and in Section 4, we complete the proof of Theorem \ref{th3.2}.

\section{Preliminaries}
In this section, we introduce some conventions and notations, and state some basic results of variable function spaces.  Let $\mathcal{S}(\mathbb{R}^n)$ be the Schwartz class of rapidly decreasing functions on $\mathbb{R}^n$, and $\mathcal{S}'(\mathbb{R}^n)$ the space of tempered
distributions.  Given  $f\in\mathcal{S}(\mathbb{R}^{n})$, the Fourier transform $\mathcal{F}(f)$ (or $\widehat{f}$) is defined by
$$
  \mathcal{F}(f)(\xi)=\widehat{f}(\xi):=\frac{1}{(2\pi)^{\frac{n}{2}}}\int_{\mathbb{R}^{n}}f(x)e^{-ix\cdot\xi}dx.
$$
More generally, the Fourier transform of  a tempered distribution $f\in\mathcal{S}'(\mathbb{R}^{n})$ is defined by the dual argument in the standard way. Throughout this paper,  the letters $C$ and $C_{i}$ ($i=1,2, \cdots$) stand for the generic harmless constants, whose meaning is clear from the context. For brevity, we shall use the notation
$u\lesssim v$ instead of $u\leq Cv$, and $u\approx v$ means that $u\lesssim v$
and $v\lesssim u$.

 \subsection{Variable Lebesgue spaces}
 We begin with a fundamental definition.  Let $\mathcal{P}_{0}(\mathbb{R}^{n})$ be the set of all Lebesgue measurable functions $p(\cdot):\mathbb{R}^{n}\rightarrow[1,+\infty)$ such that
\begin{equation*}
1< p^{-}:=\operatorname{essinf}_{x\in\mathbb{R}^{n}}p(x), \ \  p^{+}:=\operatorname{esssup}_{x\in\mathbb{R}^{n}}p(x)<+\infty.
\end{equation*}
For any $p(\cdot)\in \mathcal{P}_{0}(\mathbb{R}^{n})$,  a measurable function $f$, we denote
\begin{equation}\label{eq2.1}
   \|f\|_{L^{p(\cdot)}}:=\inf\left\{\lambda>0:  \rho_{p(\cdot)}\left(\frac{f}{\lambda}\right)\leq 1\right\},
\end{equation}
where the modular function $\rho_{p(\cdot)}$ associated with $p(\cdot)$ is given by
\begin{equation*}
   \rho_{p(\cdot)}(f):=\int_{\mathbb{R}^{n}}|f(x)|^{p(x)}dx.
\end{equation*}
If the set on the right-hand side of  \eqref{eq2.1} is empty then we denote $\|f\|_{L^{p(\cdot)}}=\infty$. Now let us recall the definition of the Lebesgue spaces $L^{p(\cdot)}(\mathbb{R}^{n})$ with variable exponent.

\begin{definition}\label{de2.1}
Given $p(\cdot)\in \mathcal{P}_{0}(\mathbb{R}^{n})$, we define the variable Lebesgue spaces $L^{p(\cdot)}(\mathbb{R}^{n})$ to be the set of Lebesgue measurable functions $f$ such that   $\|f\|_{L^{p(\cdot)}}<+\infty$.
\end{definition}

The variable Lebesgue spaces $L^{p(\cdot)}(\mathbb{R}^{n})$ retain some good properties of the usual Lebesgue spaces, such as $L^{p(\cdot)}(\mathbb{R}^{n})$ is a Banach space associated with the norm $\|\cdot\|_{L^{p(\cdot)}}$, and we have the following  H\"{o}lder's inequality (cf. \cite{CF13}, Corollary 2.28; \cite{DHHR11}, Lemma 3.2.20). For a complete  presentation of the theory of a  variable Lebesgue spaces, we refer to books \cite{CF13, DHHR11}.
\smallbreak

\begin{lemma}\label{le2.2}
Given two exponent functions $p_{1}(\cdot), p_{2}(\cdot)\in \mathcal{P}_{0}(\Omega)$, define $p(\cdot)\in \mathcal{P}_{0}(\Omega)$ by $\frac{1}{p(x)}=\frac{1}{p_{1}(x)}+\frac{1}{p_{2}(x)}$. Then there exists a constant $C$ such that for all $f\in L^{p_{1}(\cdot)}(\Omega)$ and $g\in L^{p_{2}(\cdot)}(\Omega)$, we have $fg\in L^{p(\cdot)}(\Omega)$ and
\begin{equation}\label{eq2.2}
   \|fg\|_{L^{p(\cdot)}}\leq C \|f\|_{L^{p_{1}(\cdot)}}\|g\|_{L^{p_{2}(\cdot)}}.
\end{equation}
\end{lemma}

In order to tackle with the boundedness of many classical operators appeared in the mathematical analysis of PDEs, a classical approach is to consider some constraints on the variable exponent, and the most common one is given by the so-called \textit{log-H\"{o}lder continuity condition}.
\begin{definition}\label{de2.3}
Let $p(\cdot)\in \mathcal{P}_{0}(\mathbb{R}^{n})$ such that  there exists a limit $\frac{1}{p_{\infty}}=\lim_{|x|\rightarrow\infty}\frac{1}{p(x)}$.
\begin{itemize}
\item We say that $p(\cdot)$ is locally log-H\"{o}lder continuous if for all $x,y\in \mathbb{R}^{n}$,  there exists a constant $C$ such that $\big{|}\frac{1}{p(x)}-\frac{1}{p(y)}\big{|}\leq \frac{C}{\log(e+\frac{1}{|x-y|})}$;
\item We say that $p(\cdot)$ satisfies the log-H\"{o}lder decay condition if for all $x\in \mathbb{R}^{n}$,  there exists a constant $C$ such that $\big{|}\frac{1}{p(x)}-\frac{1}{p_{\infty}}\big{|}\leq \frac{C}{\log(e+|x|)}$;
\item We say that  $p(\cdot)$ is globally log-H\"{o}lder continuous in $\mathbb{R}^n$ if it is locally \textit{log-H\"{o}lder continuous} and satisfies the \textit{log-H\"{o}lder decay condition};
\item We define the class of variable exponents $\mathcal{P}_{0}^{\log}(\mathbb{R}^{n})$ as
\begin{equation*}
 \mathcal{P}_{0}^{\log}(\mathbb{R}^{n}):=\left\{p(\cdot)\in\mathcal{P}(\mathbb{R}^{n}):\ \  p(\cdot)\ \text{is globally log-H\"{o}lder continuous in}\  \mathbb{R}^n\right\}.
\end{equation*}
\end{itemize}
\end{definition}

For any $p(\cdot)\in \mathcal{P}_{0}^{\log}(\mathbb{R}^{n})$, we have the following results in terms of the Hardy--Littlewood maximal function and  Riesz transforms (cf. \cite{CF13}, Theorem 3.16 and Theorem 5.42;\cite{DHHR11}, Theorem 4.3.8 and Corollary 6.3.10).

\begin{lemma}\label{le2.4}
Let $p(\cdot)\in \mathcal{P}^{\log}_{0}(\mathbb{R}^{n})$ with $1<p^{-}\leq p^{+}<+\infty$. Then  for any $f\in L^{p(\cdot)}(\mathbb{R}^{n})$, there exists a positive constant $C$ such that
\begin{equation}\label{eq2.3}
   \|\mathcal{M} (f)\|_{L^{p(\cdot)}}\leq C\|f\|_{L^{p(\cdot)}},
\end{equation}
where $\mathcal{M}$ is the Hardy--Littlewood maximal function defined by
\begin{equation*}
  \mathcal{M}(f)(x):=\sup_{x\in B}\frac{1}{|B|}\int_{B}|f(y)|dy,
\end{equation*}
and $B\subset \mathbb{R}^{n}$ is an open ball with center $x$. Furthermore,
\begin{equation}\label{eq2.4}
   \|\mathcal{R}_{j}(f)\|_{L^{p(\cdot)}}\leq C\|f\|_{L^{p(\cdot)}} \ \ \text{for any}\ \ 1\leq j\leq n,
\end{equation}
where $\mathcal{R}_{j}$ ($1\leq j\leq n$) are the usual Riesz transforms, i.e., $\mathcal{F}\left(\mathcal{R}_{j}f\right)(\xi)=-\frac{i\xi_{j}}{|\xi|}\mathcal{F}(f)(\xi)$.
\end{lemma}

\subsection{Littlewood--Paley decomposition and Fourier--Besov spaces}
Let us briefly recall some basic facts on Littlewood--Paley dyadic decomposition theory.  More details may be found in Chap. 2 and Chap. 3 in the book \cite{BCD11}. Choose a smooth radial non-increasing function $\chi$  with $\operatorname{Supp}\chi\subset B(0,\frac{4}{3})$ and $\chi\equiv1$ on $B(0,\frac{3}{4})$. Set $\varphi(\xi)=\chi(\frac{\xi}{2})-\chi(\xi)$. It is not difficult to check that
 $\varphi$ is supported in the shell $\{\xi\in\mathbb{R}^{n},\ \frac{3}{4}\leq
|\xi|\leq \frac{8}{3}\}$, and
\begin{align*}
   \sum_{j\in\mathbb{Z}}\varphi(2^{-j}\xi)=1, \ \ \  \xi\in\mathbb{R}^{n}\backslash\{0\}.
\end{align*}
 For any $f\in\mathcal{S}'(\mathbb{R}^{n})$, the homogeneous dyadic blocks $\Delta_{j}$ ($j\in\mathbb {Z}$) and the low-frequency cutoff operator  $ S_{j}$ are defined for all $j\in\mathbb{Z}$ by
\begin{align*}
 \Delta_{j}f: =\mathcal{F}^{-1}(\varphi(2^{-j}\cdot)\mathcal{F}f), \ \ \ S_{j}f(x): =\mathcal{F}^{-1}(\chi(2^{-j}\cdot)\mathcal{F}f).
\end{align*}
Let $\mathcal{S}'_{h}(\mathbb{R}^{n})$ be the space of tempered distribution $f\in\mathcal{S}'(\mathbb{R}^{n})$ such that
$$
\lim_{j\rightarrow -\infty} S_{j}f(x)=0.
$$
Then one has the unit decomposition for any tempered distribution $f\in\mathcal{S}'_{h}(\mathbb{R}^{n})$:
\begin{align}\label{eq2.5}
 f=\sum_{j\in\mathbb{Z}}\Delta_{j}f.
\end{align}
The above homogeneous dyadic block $\Delta_{j}$ and
the partial summation operator $S_{j}$  satisfy the following quasi-orthogonal properties: for any $f, g\in\mathcal{S}'(\mathbb{R}^{n})$, one has
\begin{align}\label{eq2.6}
  \Delta_{i}\Delta_{j}f\equiv0\ \ \ \text{if}\ \ \ |i-j|\geq 2\ \ \ \text{and}\ \ \
  \Delta_{i}(S_{j-1}f\Delta_{j}g)\equiv0 \ \ \ \text{if}\ \ \ |i-j|\geq 5.
\end{align}
\smallbreak

Applying the above decomposition, the homogeneous Fourier--Besov space $F\dot{B}_{p,r}^s(\mathbb{R}^{n})$ and the Chemin--Lerner type space $ \tilde{L}^{\lambda}(0,T; F\dot{B}^{s}_{p,r}(\mathbb{R}^{n}))$ can be defined as follows:

\begin{definition}\label{def2.5}
Let $s\in \mathbb{R}$ and $1\leq p,r\leq\infty$, the homogeneous Fourier--Besov space $F\dot{B}^{s}_{p,r}(\mathbb{R}^{n})$ is defined to be the set of all tempered distributions $f\in \mathcal{S}'_{h}(\mathbb{R}^{n})$ such that $\hat{f} \in L_{loc}^1(\mathbb{R}^n)$ and the following norm is finite:
\begin{equation*}
  \|f\|_{F\dot{B}^{s}_{p,r}}:= \begin{cases} \left(\sum_{j\in\mathbb{Z}}2^{jsr}\|\widehat{\Delta_{j}f}\|_{L^{p}}^{r}\right)^{\frac{1}{r}}
  \ \ &\text{if}\ \ 1\leq r<\infty,\\
  \sup_{j\in\mathbb{Z}}2^{js}\|\widehat{\Delta_{j}f}\|_{L^{p}}\ \
  &\text{if}\ \
  r=\infty.
 \end{cases}
\end{equation*}
\end{definition}

\begin{definition}\label{def2.6} For $0<T\leq\infty$, $s\in \mathbb{R}$ and $1\leq p, r, \lambda\leq\infty$, we set  {\rm{(}}with the usual convention if
$r=\infty${\rm{)}}:
$$
  \|f\|_{\mathcal{L}^{\lambda}_{T}(F\dot{B}^{s}_{p,r})}:=\big(\sum_{j\in\mathbb{Z}}2^{jsr}\|\widehat{\Delta_{j}f}\|_{L^{\lambda}_{T}(L^{p})}^{r}\big)^{\frac{1}{r}}.
$$
Then we define the space $\mathcal{L}^{\lambda}(0,T; F\dot{B}^{s}_{p,r}(\mathbb{R}^{n}))$  as the set of temperate distributions $f$ over $(0,T)\times \mathbb{R}^{n}$ such that $\lim_{j\rightarrow -\infty}S_{j}f=0$ in $\mathcal{S}'((0,T)\times\mathbb{R}^{n})$ and $\|f\|_{\mathcal{L}^{\lambda}_{T}(F\dot{B}^{s}_{p,r})}<\infty$.
\end{definition}

Let us state some basic properties of the homogeneous Fourier--Besov spaces (cf. Lemma 2.5 in \cite{ZZ19}).
\begin{lemma}\label{le2.7}
The following properties hold in the homogeneous Fourier--Besov spaces:
\begin{itemize}
  \item[(\romannumeral1)] There exists a universal constant $C$ such that
    \begin{align*}
     C^{-1}\|u\|_{F\dot{B}^{s}_{p,r}} \leq  \|\nabla u\|_{F\dot{B}^{s-1}_{p,r}} \leq C\|u\|_{F\dot{B}^{s}_{p,r}}.
    \end{align*}
  \item[(\romannumeral2)] For $1\leq p_1 \leq p_2\leq\infty$ and $1\leq r_1 \leq r_2\leq\infty$, we have the continuous imbedding $F\dot{B}^{s}_{p_2,r_2}(\mathbb{R}^n) \hookrightarrow F\dot{B}^{s-n(\frac{1}{p_1}-\frac{1}{p_2})}_{p_1,r_1}(\mathbb{R}^n)$.
  \item[(\romannumeral3)] For $s_1,s_2 \in \mathbb{R}$ such that $s_1<s_2$ and $\theta\in (0,1)$, there exists a constant $C$ such that
    \begin{align*}
     \|u\|_{F\dot{B}^{s_1\theta+s_2(1-\theta)}_{p,r}} \leq  C\|u\|^{\theta}_{F\dot{B}^{s_1}_{p,r}} \|u\|^{1-\theta}_{F\dot{B}^{s_2}_{p,r}}.
    \end{align*}
\end{itemize}
\end{lemma}

We also need to establish the following generalized product estimates between two
distributions in the homogeneous Fourier--Besov spaces, for details, see Lemmas 2.9 and 2.10 in \cite{VZ242}.

\begin{lemma}\label{le2.8}
Let $s>0$ and $1 \leq p, p_1, p_2 \leq \infty$ such that
\begin{equation*}
  1+\frac{1}{p}=\frac{1}{p_{1}}+\frac{1}{p_{2}}.
\end{equation*}
Then there exists a constant $C>0$ such that
\begin{equation}\label{eq2.7}
\|uv\|_{F\dot{B}_{p,1}^{s}} \leq C\left( \|u\|_{F\dot{B}_{p_1,1}^{s}} \|v\|_{F\dot{B}_{p_2,1}^{0}}+\|u\|_{F\dot{B}_{p_2,1}^{0}} \|v\|_{F\dot{B}_{p_1,1}^{s}}\right).
\end{equation}
\end{lemma}

\begin{lemma}\label{le2.9}
Let $s_1, s_2 \in \mathbb{R}$ and $1 \leq p_1,p_2 \leq \infty$.  Assume that
\begin{equation*}
   s_1 \leq n\min\{1-\frac{1}{p_1}, 1-\frac{1}{p_2}\}, \ \ s_2\leq n(1-\frac{1}{p_2})
\end{equation*}
and
\begin{equation*}
  s_1+s_2>\max\{0, n(1-\frac{1}{p_1}-\frac{1}{p_2})\}.
\end{equation*}
Then there exists a constant $C>0$ such that
\begin{equation}\label{eq2.11}
\|uv\|_{F\dot{B}_{p_1,1}^{s_1+s_2-n(1-\frac{1}{p_2})}} \leq C \|u\|_{F\dot{B}_{p_1,1}^{s_1}} \|v\|_{F\dot{B}_{p_2,1}^{s_2}}.
\end{equation}
\end{lemma}

\subsection{Variable Fourier--Besov spaces}
 The Besov spaces $\dot{B}^{s(\cdot)}_{p(\cdot),r}(\mathbb{R}^{n})$ with variable regularity and integrability were introduced  by Vyb\'{r}al \cite{V09},  Kempka \cite{K09} and Almeida--H\"{a}st\"{o} \cite{AH10}. Inspired by the ideas in \cite{AH10},  Ru \cite{R18} introduced the variable Fourier--Besov spaces to establish the global well-posedness of the 3D Navier--Stokes equations in such spaces.  Let us recall its  definition. Let $p(\cdot),r(\cdot)\in \mathcal{P}_{0}(\mathbb{R}^{n})$. The mixed Lebesgue-sequence space $\ell^{r(\cdot)}(L^{p(\cdot)})$ is defined on sequences of $L^{p(\cdot)}$-functions by the modular
\begin{equation}\label{eq2.17}
 \varrho_{\ell^{r(\cdot)}(L^{p(\cdot)})}\left(\{f_{j}\}_{j\in\mathbb{Z}}\right):=\sum_{j\in \mathbb{Z}}\inf\left\{\lambda_{j}>0\Big{|} \varrho_{p(\cdot)}\left(\frac{f_{j}}{\lambda_{j}^{\frac{1}{r(\cdot)}}}\right) \leq 1 \right\}.
\end{equation}
The norm is defined from this as usual:
\begin{equation*}
 \left\|\{f_{j}\}_{j\in\mathbb{Z}}\right\|_{\ell^{r(\cdot)}(L^{p(\cdot)})}:=\inf\left\{\mu>0\Big{|} \varrho_{\ell^{r(\cdot)}(L^{p(\cdot)})}\left(\{\frac{f_{j}}{\mu}\}_{j\in\mathbb{Z}}\right) \leq 1 \right\}.
\end{equation*}
Since we assume that $r^{+}<\infty$,  it holds that
\begin{equation*}
 \varrho_{\ell^{r(\cdot)}(L^{p(\cdot)})}\left(\{f_{j}\}_{j\in\mathbb{Z}}\right)=\sum_{j\in\mathbb{Z}}\left\||f_{j}|^{r(\cdot)}\right\|_{\frac{p(\cdot)}{r(\cdot)}}.
\end{equation*}

Let us recall the definition of the  Fourier--Besov space $F\dot{B}^{s(\cdot)}_{p(\cdot),r(\cdot)}(\mathbb{R}^{n})$ with variable exponents.

\begin{definition}\label{de2.10}
Let $s(\cdot), p(\cdot),r(\cdot)\in \mathcal{P}_{0}^{\log}(\mathbb{R}^{n})$. We define the variable exponent Fourier--Besov space $F\dot{B}^{s(\cdot)}_{p(\cdot),r(\cdot)}(\mathbb{R}^{n})$  as the collection of all distributions $f\in \mathcal{S}'_{h}(\mathbb{R}^{n})$ such that
\begin{equation}\label{eq2.10}
\left\|f\right\|_{F\dot{B}^{s(\cdot)}_{p(\cdot),r(\cdot)}}:=\left\|\{2^{js(\cdot)}\varphi_{j}\widehat{f}\}_{j\in\mathbb{Z}}\right\|_{\ell^{r(\cdot)}(L^{p(\cdot)}_{\xi})}<\infty.
\end{equation}
\end{definition}

We also need the following Chemin--Lerner type mixed time-space spaces in terms of the Fourier--Besov space with variable exponent.

\begin{definition}\label{de2.11}
Let $s(\cdot), p(\cdot)\in \mathcal{P}_{0}^{\log}(\mathbb{R}^{n})$ and  $1\leq \rho, r\leq\infty$. We define the Chemin--Lerner mixed time-space $\mathcal{L}^{\rho}(0,T; \dot{B}^{s(\cdot)}_{p(\cdot),r}(\mathbb{R}^{n}))$ as the completion of $\mathcal{C}([0,T]; \mathcal{S}(\mathbb{R}^{n}))$ by the norm
\begin{equation}\label{eq2.11}
\left\|f\right\|_{\mathcal{L}^{\rho}_{T}(F\dot{B}^{s(\cdot)}_{p(\cdot),r})}
:=\left(\sum_{j\in\mathbb{Z}}\left\|2^{js(\cdot)}\varphi_{j}\widehat{f}\right\|^{r}_{L^{\rho}_{T}(L^{p(\cdot)}_{\xi})}\right)^{\frac{1}{r}}<\infty
\end{equation}
with the usual change if $\rho=\infty$
or $r=\infty$. If $T=\infty$, we use $\left\|f\right\|_{\mathcal{L}^{\rho}_{t}(F\dot{B}^{s(\cdot)}_{p(\cdot),r})}$ instead of $\left\|f\right\|_{\mathcal{L}^{\rho}_{\infty}(F\dot{B}^{s(\cdot)}_{p(\cdot),r})}$.
\end{definition}

Finally, let us recall the following existence and uniqueness result for an abstract operator equation in a generic Banach space (cf. \cite{L02}, Theorem 13.2).
\begin{proposition}\label{pro2.12}
Let $\mathcal{X}$ be a Banach space and
$\mathcal{B}:\mathcal{X}\times\mathcal{X}\rightarrow\mathcal{X}$  a
bilinear bounded operator. Assume that for any $u,v\in
\mathcal{X}$, we have
$$
  \|\mathcal{B}(u,v)\|_{\mathcal{X}}\leq
  C\|u\|_{\mathcal{X}}\|v\|_{\mathcal{X}}.
$$
Then for any $y\in \mathcal{X}$ such that $\|y\|_{\mathcal{X}}\leq
\eta<\frac{1}{4C}$, the equation $u=y+\mathcal{B}(u,u)$ has a
solution $u$ in $\mathcal{X}$. Moreover, this solution is the only
one such that $\|u\|_{\mathcal{X}}\leq 2\eta$, and depends
continuously on $y$ in the following sense: if
$\|\widetilde{y}\|_{\mathcal{X}}\leq \eta$,
$\widetilde{u}=\widetilde{y}+\mathcal{B}(\widetilde{u},\widetilde{u})$
and $\|\widetilde{u}\|_{\mathcal{X}}\leq 2\eta$, then
$$
  \|u-\widetilde{u}\|_{\mathcal{X}}\leq \frac{1}{1-4\eta
  C}\|y-\widetilde{y}\|_{\mathcal{X}}.
$$
\end{proposition}

\section{Proof of Theorem \ref{th3.1}}

In order to apply Proposition \ref{pro2.12} to the generalized Navier--Stokes equations \eqref{eq3.1}, we need to get rid of the pressure $\pi$, and for this purpose we apply to \eqref{eq3.1} the Leray projector $\mathbb{P}$ defined by $\mathbb{P}u=u+\nabla(-\Delta)^{-1}\nabla\cdot u$, i.e., $\mathbb{P}$ is the $3\times 3$
matrix pseudo-differential operator with the
symbol $(\delta_{ij}-\frac{\xi_i\xi_j}{|\xi|^2})_{i,j=1}^3$ with
$\delta_{ij}=1$ if $i= j$, and $\delta_{ij}=0$ if $i\neq j$.  Recall that we have $\mathbb{P}(\nabla\pi)=0$, and if a vector field $u$ is divergence free then we have the identity $\mathbb{P}(u)=u$. Thus we rewrite the generalized Navier--Stokes equations \eqref{eq3.1} as
 \begin{equation} \label{eq4.1}
\left\{
\begin{aligned}
& \partial_t u+\Lambda^{\alpha}u +\mathbb{P} \nabla\cdot(u\otimes u)=\mathbb{P}f,  \\
& \nabla\cdot u=0,\\
& u(0,x) =u_0(x).
\end{aligned}
\right.
\end{equation}
Based on the framework of the Kato's analytical semigroup, we can further reduce the system \eqref{eq4.1} into an equivalent integral form:
\begin{equation}\label{eq4.2}
  u(t,x)=e^{-t\Lambda^{\alpha}}u_{0}(x)+\int_{0}^{t}e^{-(t-\tau)\Lambda^{\alpha}}\mathbb{P}fd\tau-\int_{0}^{t}e^{-(t-\tau)\Lambda^{\alpha}}\mathbb{P}\nabla\cdot(u\otimes u)d\tau.
\end{equation}
For simplicity, we denote the third term on the right-hand side of \eqref{eq4.2}  as
\begin{equation*}
  \mathcal{B}_{1}(u,u):=\int_{0}^{t}e^{-(t-\tau)\Lambda^{\alpha}}\mathbb{P}\nabla\cdot(u\otimes u)d\tau.
\end{equation*}

Now,  let the assumptions of Theorem \ref{th3.1} be in force, we introduce the solution space $\mathcal{X}_{t}$ as
$$
\mathcal{X}_{t}:=\mathcal{L}^{\rho}(0,\infty; F\dot{B}^{4-\alpha-\frac{3}{p(\cdot)}+\frac{\alpha}{\rho}}_{p(\cdot),1}(\mathbb{R}^{3}))\cap \mathcal{L}^{\infty}(0,\infty; F\dot{B}^{\frac{5}{2}-\alpha}_{2,1}(\mathbb{R}^{3}))\cap \mathcal{L}^{1}(0,\infty; F\dot{B}^{\frac{5}{2}}_{2,1}(\mathbb{R}^{3}))
$$
and the norm of the  space $\mathcal{X}_{t}$ is endowed by
\begin{equation*}
   \|u\|_{\mathcal{X}_{t}}=\max\big{\{}\|u\|_{\mathcal{L}^{\rho}_{t}( F\dot{B}^{4-\alpha-\frac{3}{p(\cdot)}+\frac{\alpha}{\rho}}_{p(\cdot),1})}, \|u\|_{\mathcal{L}^{\infty}_{t}(F\dot{B}^{\frac{5}{2}-\alpha}_{2,1})}, \|u\|_{\mathcal{L}^{1}_{t}( F\dot{B}^{\frac{5}{2}}_{2,1})} \big{\}}.
\end{equation*}
It can be easily seen that $(\mathcal{X}_{t},  \|\cdot\|_{\mathcal{X}_{t}})$ is a Banach space.
\smallbreak

Let us establish the linear and nonlinear estimates of the integral system \eqref{eq4.2} in the space $\mathcal{X}_{t}$, respectively.
\begin{lemma}\label{le4.1}
Let the assumptions of Theorem \ref{th3.1} be in force. Then for any $u_{0}\in F\dot{B}^{4-\alpha-\frac{3}{p(\cdot)}}_{p(\cdot),1}(\mathbb{R}^{3})$,
there exists a positive constant $C$ such that
\begin{equation}\label{eq4.3}
   \left\|e^{t\Delta} u_{0}\right\|_{\mathcal{X}_{t}}\leq
   C\left\|u_{0}\right\|_{F\dot{B}^{4-\alpha-\frac{3}{p(\cdot)}}_{p(\cdot),1}}.
\end{equation}
\end{lemma}
\begin{proof}
We split the proof into the following steps.

\textbf{Step 1}.  Applying Theorem \ref{th1.1}, one can easily see that
\begin{equation}\label{eq4.4}
\|e^{-t\Lambda^{\alpha}} u_{0}\|_{\mathcal{L}^{\rho}_{t}(F\dot{B}^{4-\alpha-\frac{3}{p(\cdot)}+\frac{\alpha}{\rho}}_{p(\cdot),1})}\leq C\|u_{0}\|_{ F\dot{B}^{4-\alpha-\frac{3}{p(\cdot)}}_{p(\cdot),1}}.
\end{equation}

\textbf{Step 2}.  For any $1\leq \rho \leq \infty$, we can derive that
\begin{align}\label{eq4.5}
   &\left\|e^{-t\Lambda^{\alpha}} u_{0}\right\|_{\mathcal{L}^{\rho}_{t}(F\dot{B}^{\frac{5}{2}-\alpha+\frac{\alpha}{\rho}}_{2,1})}
= \sum_{j\in \mathbb{Z}}\left\|2^{j(\frac{5}{2}-\alpha+\frac{\alpha}{\rho})}\varphi_{j}e^{-t|\cdot|^{\alpha}}\widehat{u}_{0}\right\|_{L^{\rho}_{t}(L^{2}_{\xi})}\nonumber\\
   &  \lesssim\sum_{j\in \mathbb{Z}} \sum_{k=-1}^{1}\left\|2^{j(4-\alpha-\frac{3}{p(\cdot)})}\varphi_{j}\widehat{u}_{0}\right\|_{L^{p(\cdot)}_{\xi}}
   \left\|2^{j(\frac{\alpha}{\rho}-\frac{3}{2}+\frac{3}{p(\cdot)})}\varphi_{j+k}e^{-t|\cdot|^{\alpha}}\right\|_{L^{\rho}_{t}(L^{\frac{2p(\cdot)}{p(\cdot)-2}}_{\xi})}.
\end{align}
Notice that
\begin{align}\label{eq4.6}
   \left\|2^{j(\frac{\alpha}{\rho}-\frac{3}{2}+\frac{3}{p(\cdot)})}\varphi_{j+k}e^{-t|\cdot|^{\alpha}}\right\|_{L^{\rho}_{t}(L^{\frac{2p(\cdot)}{p(\cdot)-2}}_{\xi})}
   &\lesssim  \left\|2^{j\frac{\alpha}{\rho}}e^{-t2^{\alpha j}}\right\|_{L^{\rho}_{t}} \left\|2^{j(-\frac{3}{2}+\frac{3}{p(\cdot)})}\varphi_{j}\right\|_{L^{\frac{2p(\cdot)}{p(\cdot)-2}}_{\xi}}\nonumber\\
   &\lesssim\left\|2^{j(-\frac{3}{2}+\frac{3}{p(\cdot)})}\varphi_{j}\right\|_{L^{\frac{2p(\cdot)}{p(\cdot)-2}}_{\xi}},
\end{align}
and
\begin{align}\label{eq4.7}
   \left\|2^{j(-\frac{3}{2}+\frac{3}{p(\cdot)})}\varphi_{j}\right\|_{L^{\frac{2p(\cdot)}{p(\cdot)-2}}_{\xi}}&= \left\{\lambda>0:\ \int_{\mathbb{R}^{3}}\left|\frac{2^{j(-\frac{3}{2}+\frac{3}{p(\xi)})}\varphi_{j}(\xi)}{\lambda}\right|^{\frac{2p(\xi)}{p(\xi)-2}}d\xi<1\right\}\nonumber\\
   &=\left\{\lambda>0:\ \int_{\mathbb{R}^{3}}\left|\frac{\varphi_{j}(\xi)}{\lambda}\right|^{\frac{2p(\xi)}{p(\xi)-2}}2^{-3j}d\xi<1\right\}\nonumber\\
   &\lesssim \left\{\lambda>0:\ \int_{\mathbb{R}^{3}}\left|\frac{\varphi_{j}(2^{j}\xi)}{\lambda}\right|^{\frac{2p(2^{j}\xi)}{p(2^{j}\xi)-2}}d\xi<1\right\}\nonumber\\
   &\lesssim 1.
\end{align}
Taking the above two estimates \eqref{eq4.6} and \eqref{eq4.7} into
\eqref{eq4.5}, we obtain
\begin{align*}
   \left\|e^{-t\Lambda^{\alpha}} u_{0}\right\|_{\mathcal{L}^{\rho}_{t}(\dot{B}^{\frac{5}{2}-\alpha+\frac{\alpha}{\rho}}_{2,1})}
      \lesssim\sum_{j\in \mathbb{Z}}\left\|2^{j(4-\alpha-\frac{3}{p(\cdot)})}\varphi_{j}\widehat{u}_{0}\right\|_{L^{p(\cdot)}_{\xi}}
   \lesssim\|u_{0}\|_{F\dot{B}^{4-\alpha-\frac{3}{p(\cdot)}}_{p(\cdot),1}}.
\end{align*}
Then taking $\rho=1$ and $\rho=\infty$ in above inequality, respectively, we get
\begin{align}\label{eq4.8}
   \left\|e^{-t\Lambda^{\alpha}} u_{0}\right\|_{\mathcal{L}^{1}_{t}(\dot{B}^{\frac{5}{2}}_{2,1})}+ \left\|e^{-t\Lambda^{\alpha}} u_{0}\right\|_{\mathcal{L}^{\infty}_{t}(\dot{B}^{\frac{5}{2}-\alpha}_{2,1})}
         \lesssim\|u_{0}\|_{F\dot{B}^{4-\alpha-\frac{3}{p(\cdot)}}_{p(\cdot),1}}.
\end{align}
We complete the proof of Lemma \ref{le4.1}.
\end{proof}
\smallbreak

\begin{lemma}\label{le4.2}
Let the assumptions of Theorem \ref{th3.1} be in force. Then for any $f\in\mathcal{L}^{1}(0,\infty; F\dot{B}^{4-\alpha-\frac{3}{p(\cdot)}}_{p(\cdot),1}(\mathbb{R}^{3}))\cap \mathcal{L}^{1}(0,\infty; F\dot{B}^{\frac{5}{2}-\alpha}_{2,1}(\mathbb{R}^{3}))$, there exists a positive constant $C$ such that
\begin{equation}\label{eq4.9}
   \left\|\int_{0}^{t}e^{-(t-\tau)\Lambda^{\alpha}}\mathbb{P}fd\tau\right\|_{\mathcal{X}_{t}}\leq
   C\left(\|f\|_{\mathcal{L}^{1}_{t}( F\dot{B}^{4-\alpha-\frac{3}{p(\cdot)}}_{p(\cdot),1})}+\|f\|_{\mathcal{L}^{1}_{t}( F\dot{B}^{\frac{5}{2}-\alpha}_{2,1})}\right).
\end{equation}
\end{lemma}
\begin{proof}
Since the Leray projector $\mathbb{P}$ is bounded under the Fourier variables, we can exactly  obtain \eqref{eq4.9} by Theorem \ref{th1.1}.
\end{proof}

\begin{lemma}\label{le4.3}
Let the assumptions of Theorem \ref{th3.1} be in force.  Then
there exists a positive constant $C$ such that
\begin{equation}\label{eq4.10}
   \left\|\mathcal{B}_{1}(u,u)\right\|_{\mathcal{X}_{t}}\leq C
    \|u\|_{\mathcal{X}_{t}}^{2}.
\end{equation}
\end{lemma}
\begin{proof}
We split the proof into the following two steps.

\textbf{Step 1}.  According to Theorem \ref{th1.1}, it is clear that
\begin{align}\label{eq4.11}
   \left\|\mathcal{B}_{1}(u,u)\right\|_{\mathcal{L}^{\rho}_{t}(F\dot{B}^{4-\alpha-\frac{3}{p(\cdot)}+\frac{\alpha}{\rho}}_{p(\cdot),1})}&\lesssim \left\|\nabla\cdot(u\otimes u)\right\|_{\mathcal{L}^{1}_{t}(F\dot{B}^{4-\alpha-\frac{3}{p(\cdot)}}_{p(\cdot),1})}\nonumber\\&\lesssim \left\|u\otimes u\right\|_{\mathcal{L}^{1}_{t}(F\dot{B}^{5-\alpha-\frac{3}{p(\cdot)}}_{p(\cdot),1})}.
\end{align}
Furthermore, we can estimate the right-hand side of \eqref{eq4.11} as
\begin{align}\label{eq4.12}
   \left\|u\otimes u\right\|_{\mathcal{L}^{1}_{t}(F\dot{B}^{5-\alpha-\frac{3}{p(\cdot)}}_{p(\cdot),1})}
   &=\sum_{j\in\mathbb{Z}}\left\|2^{j(5-\alpha-\frac{3}{p(\cdot)})}\varphi_{j}\mathcal{F}\left(u\otimes u\right)\right\|_{L^{1}_{t}(L^{p(\cdot)}_{\xi})}\nonumber\\
   &\lesssim\sum_{j\in\mathbb{Z}}\sum_{k=-1}^{1}\left\|2^{-3j(\frac{6-(5-2\alpha)p(\cdot)}{6p(\cdot)})}\varphi_{j}\right\|_{L^{\frac{6p(\cdot)}{6-(5-2\alpha)p(\cdot)}}_{\xi}}
   \left\|2^{\frac{5}{2}j}\varphi_{j+k}\mathcal{F}\left(u\otimes u\right)\right\|_{L^{1}_{t}L^{\frac{6}{5-2\alpha}}_{\xi}}\nonumber\\
  \nonumber\\
   &\lesssim \left\|2^{\frac{5}{2}j}\varphi_{j+k}\mathcal{F}\left(u\otimes u\right)\right\|_{L^{1}_{t}L^{\frac{6}{5-2\alpha}}_{\xi}}\nonumber\\
   &\approx  \left\|u\otimes u\right\|_{\mathcal{L}^{1}_{t}(F\dot{B}^{\frac{5}{2}}_{\frac{6}{5-2\alpha},1})}.
\end{align}
By using Lemma \ref{le2.8} with $p=\frac{6}{5-2\alpha}$, $p_{1}=2$ and $p_{2}=\frac{3}{4-\alpha}$, we obtain that
\begin{align}\label{eq4.13}
 \left\|u\otimes u\right\|_{\mathcal{L}^{1}_{t}(F\dot{B}^{\frac{5}{2}}_{\frac{6}{5-2\alpha},1})}
 &\lesssim  \left\|u\right\|_{\mathcal{L}^{1}_{t}(F\dot{B}^{\frac{5}{2}}_{2,1})}\left\|u\right\|_{\mathcal{L}^{\infty}_{t}(F\dot{B}^{0}_{\frac{3}{4-\alpha},1})}\nonumber\\
 &\lesssim  \left\|u\right\|_{\mathcal{L}^{1}_{t}(F\dot{B}^{\frac{5}{2}}_{2,1})}\left\|u\right\|_{\mathcal{L}^{\infty}_{t}(F\dot{B}^{\frac{5}{2}-\alpha}_{2,1})},
\end{align}
where we have used the Sobolev embedding $F\dot{B}^{\frac{5}{2}-\alpha}_{2,1}(\mathbb{R}^{3})\hookrightarrow F\dot{B}^{0}_{\frac{3}{4-\alpha},1}(\mathbb{R}^{3})$ in Lemma \ref{le2.7}. Taking the above estimates \eqref{eq4.12} and \eqref{eq4.13} into \eqref{eq4.11}, we get
\begin{align}\label{eq4.14}
   \left\|\mathcal{B}_{1}(u,u)\right\|_{\mathcal{L}^{\rho}_{t}(F\dot{B}^{4-\alpha-\frac{3}{p(\cdot)}+\frac{\alpha}{\rho}}_{p(\cdot),1})}&\lesssim \left\|u\right\|_{\mathcal{L}^{1}_{t}(F\dot{B}^{\frac{5}{2}}_{2,1})}\left\|u\right\|_{\mathcal{L}^{\infty}_{t}(F\dot{B}^{\frac{5}{2}-\alpha}_{2,1})}.
\end{align}

\textbf{Step 2}. Applying Theorem \ref{th1.1} and Lemma \ref{le2.9}  with $s_{1}=\frac{5}{2}$,  $s_{2}=\frac{5}{2}-\alpha$ and $p_{1}=p_{2}=2$ (the indices of the time variable obey the law of H\"{o}lder's inequality), we obtain
\begin{align}\label{eq4.15}
 \left\|\mathcal{B}_{1}(u,u)\right\|_{\mathcal{L}^{1}_{t}(F\dot{B}^{\frac{5}{2}}_{2,1})}+ \left\|\mathcal{B}(u,u)\right\|_{\mathcal{L}^{\infty}_{t}(F\dot{B}^{\frac{5}{2}-\alpha}_{2,1})}&\leq \left\|u\otimes u\right\|_{\mathcal{L}^{1}_{t}(F\dot{B}^{\frac{7}{2}-\alpha}_{2,1})}\nonumber\\
   &\leq C \left\|u\right\|_{\mathcal{L}^{1}_{t}(F\dot{B}^{\frac{5}{2}}_{2,1})}\left\|u\right\|_{\mathcal{L}^{\infty}_{t}(F\dot{B}^{\frac{5}{2}-\alpha}_{2,1})}.
\end{align}
Putting the above estimates \eqref{eq4.14} and \eqref{eq4.15} together, we get \eqref{eq4.10}.  Thus, we complete the proof of Lemma \ref{le4.3}.
\end{proof}
\smallbreak

Based on the desired linear and nonlinear estimates obtained in Lemmas \ref{le4.1}--\ref{le4.3}, we know that there exist two positive
constants $C_{1}$ and $C_{2}$ such that
\begin{align*}
  \|u\|_{\mathcal{X}_{t}}\leq
 C_{1}\big(\|u_{0}\|_{F\dot{B}^{4-\alpha-\frac{3}{p(\cdot)}}_{p(\cdot),1}}+\|f\|_{\mathcal{L}^{1}_{t}(F\dot{B}^{4-\alpha-\frac{3}{p(\cdot)}}_{p(\cdot),1})}
 +\|f\|_{\mathcal{L}^{1}_{t}(F\dot{B}^{\frac{5}{2}-\alpha}_{2,1})}\big)+C_{2}\|u\|_{\mathcal{X}_{t}}^2.
\end{align*}
Thus if the initial data $u_{0}$ and $f$ satisfy the condition
$$
\|u_{0}\|_{F\dot{B}^{4-\alpha-\frac{3}{p(\cdot)}}_{p(\cdot),1}}+\|f\|_{\mathcal{L}^{1}_{t}(F\dot{B}^{4-\alpha-\frac{3}{p(\cdot)}}_{p(\cdot),1})}
+\|f\|_{\mathcal{L}^{1}_{t}(F\dot{B}^{\frac{5}{2}-\alpha}_{2,1})}\leq \frac{1}{4C_{1}C_{2}},
$$
we can apply Proposition \ref{pro2.12} to get global existence of solution $u\in \mathcal{X}_{t}$ to the equations \eqref{eq3.1}. We complete the proof of Theorem \ref{th3.1}.

\section{Proof of Theorem \ref{th3.2}}

Based on the framework of the Kato's analytical semigroup, we can rewrite system \eqref{eq3.4} as an equivalent integral form:
\begin{equation}\label{eq5.1}
  u(t,x)=e^{-t\Lambda^{\alpha}}u_{0}(x)+\int_{0}^{t}e^{-(t-\tau)\Lambda^{\alpha}}fd\tau+\int_{0}^{t}e^{-(t-\tau)\Lambda^{\alpha}}\nabla\cdot(u\nabla(-\Delta)^{-1}u)d\tau.
\end{equation}
For simplicity, we denote the third term on the right-hand side of \eqref{eq5.1}  as
\begin{equation*}
  \mathcal{B}_{2}(u,u):=\int_{0}^{t}e^{-(t-\tau)\Lambda^{\alpha}}\nabla\cdot(u\nabla(-\Delta)^{-1}u)d\tau.
\end{equation*}

Now,  let the assumptions of Theorem \ref{th3.2} be in force, we introduce the solution space $\mathcal{Y}_{t}$ as follows:
$$
\mathcal{Y}_{t}:=\mathcal{L}^{\rho}(0,\infty; F\dot{B}^{3-\alpha-\frac{3}{p(\cdot)}+\frac{\alpha}{\rho}}_{p(\cdot),1}(\mathbb{R}^{3}))\cap \mathcal{L}^{1}(0,\infty; F\dot{B}^{\frac{3}{2}}_{2,1}(\mathbb{R}^{3}))\cap \mathcal{L}^{\infty}(0,\infty; F\dot{B}^{\frac{3}{2}-\alpha}_{2,1}(\mathbb{R}^{3}))
$$
and the norm of the  space $\mathcal{Y}_{t}$ is endowed by
\begin{equation*}
   \|u\|_{\mathcal{Y}_{t}}=\max\big{\{}\|u\|_{\mathcal{L}^{\rho}_{t}( F\dot{B}^{3-\alpha-\frac{3}{p(\cdot)}+\frac{\alpha}{\rho}}_{p(\cdot),1})}, \|u\|_{\mathcal{L}^{1}_{t}( F\dot{B}^{\frac{3}{2}}_{2,1})}, \|u\|_{\mathcal{L}^{\infty}_{t}(F\dot{B}^{\frac{3}{2}-\alpha}_{2,1})} \big{\}}.
\end{equation*}
It can be easily seen that $(\mathcal{Y}_{t},  \|\cdot\|_{\mathcal{Y}_{t}})$ is a Banach space.
\smallbreak

In the sequel, we shall establish the linear and nonlinear estimates of the integral equation \eqref{eq5.1} in the space $\mathcal{Y}_{t}$, respectively.
\begin{lemma}\label{le5.1}
Let the assumptions of Theorem \ref{th3.2} be in force. Then for any $u_{0}\in F\dot{B}^{3-\alpha-\frac{3}{p(\cdot)}}_{p(\cdot),1}(\mathbb{R}^{3})$,
there exists a positive constant $C$ such that
\begin{equation}\label{eq5.2}
   \|e^{t\Delta} u_{0}\|_{\mathcal{Y}_{t}}\leq
   C\|u_{0}\|_{F\dot{B}^{3-\alpha-\frac{3}{p(\cdot)}}_{p(\cdot),1}}.
\end{equation}
\end{lemma}
\begin{proof}
We split the proof into the following two steps.

\textbf{Step 1}.  Applying Theorem \ref{th1.1}, one can easily see that
\begin{equation}\label{eq5.3}
\|e^{-t\Lambda^{\alpha}} u_{0}\|_{\mathcal{L}^{\rho}_{t}(F\dot{B}^{3-\alpha-\frac{3}{p(\cdot)}+\frac{\alpha}{\rho}}_{p(\cdot),1})}\leq C\|u_{0}\|_{ F\dot{B}^{3-\alpha-\frac{3}{p(\cdot)}}_{p(\cdot),1}}.
\end{equation}

\textbf{Step 2}. For any $1\leq \rho \leq \infty$, we can use the facts \eqref{eq4.6} and \eqref{eq4.7} to derive that
\begin{align}\label{eq5.4}
   \left\|e^{-t\Lambda^{\alpha}} u_{0}\right\|_{\mathcal{L}^{\rho}_{t}(F\dot{B}^{\frac{3}{2}-\alpha+\frac{\alpha}{\rho}}_{2,1})}
   &=\sum_{j\in \mathbb{Z}}\left\|2^{j(\frac{3}{2}-\alpha+\frac{\alpha}{\rho})}\varphi_{j}e^{-t|\cdot|^{\alpha}}\widehat{u}_{0}\right\|_{L^{\rho}_{t}(L^{2}_{\xi})}\nonumber\\
   &\lesssim\sum_{j\in \mathbb{Z}} \sum_{k=-1}^{1}\left\|2^{j(3-\alpha-\frac{3}{p(\cdot)})}\varphi_{j}\widehat{u}_{0}\right\|_{L^{p(\cdot)}_{\xi}}
   \left\|2^{j(\frac{\alpha}{\rho}-\frac{3}{2}+\frac{3}{p(\cdot)})}\varphi_{j+k}e^{-t|\cdot|^{\alpha}}\right\|_{L^{\rho}_{t}(L^{\frac{2p(\cdot)}{p(\cdot)-2}}_{\xi})}\nonumber\\
    &\lesssim\sum_{j\in \mathbb{Z}}\left\|2^{j(3-\alpha-\frac{3}{p(\cdot)})}\varphi_{j}\widehat{u}_{0}\right\|_{L^{p(\cdot)}_{\xi}}\nonumber\\
   &\lesssim\|u_{0}\|_{F\dot{B}^{3-\alpha-\frac{3}{p(\cdot)}}_{p(\cdot),1}}.
\end{align}
Taking $\rho=1$ and $\rho=\infty$ in above inequality \eqref{eq5.4}, respectively, we get
\begin{align}\label{eq5.5}
   \left\|e^{-t\Lambda^{\alpha}} u_{0}\right\|_{\mathcal{L}^{1}_{t}(\dot{B}^{\frac{3}{2}}_{2,1})}+ \left\|e^{-t\Lambda^{\alpha}} u_{0}\right\|_{\mathcal{L}^{\infty}_{t}(\dot{B}^{\frac{3}{2}-\alpha}_{2,1})}
         \lesssim\|u_{0}\|_{F\dot{B}^{3-\alpha-\frac{3}{p(\cdot)}}_{p(\cdot),1}}.
\end{align}
We complete the proof of Lemma \ref{le5.1}.
\end{proof}

\begin{lemma}\label{le5.2}
Let the assumptions of Theorem \ref{th3.2} be in force. Then for any $f\in\mathcal{L}^{1}(0,\infty; F\dot{B}^{3-\alpha-\frac{3}{p(\cdot)}}_{p(\cdot),1}(\mathbb{R}^{3}))\cap \mathcal{L}^{1}(0,\infty; F\dot{B}^{\frac{3}{2}-\alpha}_{2,1}(\mathbb{R}^{3}))$, there exists a positive constant $C$ such that
\begin{equation}\label{eq5.6}
   \left\|\int_{0}^{t}e^{-(t-\tau)\Lambda^{\alpha}}fd\tau\right\|_{\mathcal{Y}_{t}}\leq
   C\left(\|f\|_{\mathcal{L}^{1}_{t}(F\dot{B}^{3-\alpha-\frac{3}{p(\cdot)}}_{p(\cdot),1})}+\|f\|_{\mathcal{L}^{1}_{t}(F\dot{B}^{\frac{3}{2}-\alpha}_{2,1})}\right).
\end{equation}
\end{lemma}
\begin{proof}
This can be exactly obtained by Theorem \ref{th1.1}.
\end{proof}

\begin{lemma}\label{le5.3}
Let the assumptions of Theorem \ref{th3.2} be in force.  Then
there exists a positive constant $C$ such that
\begin{equation}\label{eq5.7}
   \left\|\mathcal{B}_{2}(u,u)\right\|_{\mathcal{Y}_{t}}\leq C
    \|u\|_{\mathcal{Y}_{t}}^{2}.
\end{equation}
\end{lemma}
\begin{proof}We split the proof into the following two steps.

\textbf{Step 1}.  According to Theorem \ref{th1.1}, it is clear that
\begin{align}\label{eq5.8}
   \left\|\mathcal{B}_{2}(u,u)\right\|_{\mathcal{L}^{\rho}_{t}(F\dot{B}^{3-\alpha-\frac{3}{p(\cdot)}+\frac{\alpha}{\rho}}_{p(\cdot),1})}&\lesssim \left\|\nabla\cdot(u\nabla(-\Delta)^{-1}u)\right\|_{\mathcal{L}^{1}_{t}(F\dot{B}^{3-\alpha-\frac{3}{p(\cdot)}}_{p(\cdot),1})}\nonumber\\&\approx \left\|u\nabla(-\Delta)^{-1}u\right\|_{\mathcal{L}^{1}_{t}(F\dot{B}^{4-\alpha-\frac{3}{p(\cdot)}}_{p(\cdot),1})}.
\end{align}
Furthermore, we can estimate the right-hand side of \eqref{eq5.8} as
\begin{align}\label{eq5.9}
   &\left\|u\nabla(-\Delta)^{-1}u\right\|_{\mathcal{L}^{1}_{t}(F\dot{B}^{4-\alpha-\frac{3}{p(\cdot)}}_{p(\cdot),1})}
   =\sum_{j\in\mathbb{Z}}\left\|2^{j(4-\alpha-\frac{3}{p(\cdot)})}\varphi_{j}\mathcal{F}\left(u\nabla(-\Delta)^{-1}u\right)\right\|_{L^{1}_{t}(L^{p(\cdot)}_{\xi})}\nonumber\\
   &\lesssim\sum_{j\in\mathbb{Z}}\sum_{k=-1}^{1}\left\|2^{-3j(\frac{6-(5-2\alpha)p(\cdot)}{6p(\cdot)})}\varphi_{j}\right\|_{L^{\frac{6p(\cdot)}{6-p(\cdot)}}_{\xi}}
   \left\|2^{\frac{3}{2}j}\varphi_{j+k}\mathcal{F}\left(u\nabla(-\Delta)^{-1}u\right)\right\|_{L^{1}_{t}L^{\frac{6}{5-2\alpha}}_{\xi}}\nonumber\\
  \nonumber\\
   &\lesssim \left\|2^{\frac{3}{2}j}\varphi_{j+k}\mathcal{F}\left(u\nabla(-\Delta)^{-1}u\right)\right\|_{L^{1}_{t}L^{\frac{6}{5-2\alpha}}_{\xi}}\nonumber\\
   &\approx  \left\|u\nabla(-\Delta)^{-1}u\right\|_{\mathcal{L}^{1}_{t}(F\dot{B}^{\frac{3}{2}}_{\frac{6}{5-2\alpha},1})}.
\end{align}
Notice carefully that the nonlinear term $u\nabla(-\Delta)^{-1}u$ has a nice symmetric structure:
\begin{equation*}
    u\nabla(-\Delta)^{-1}u=-\nabla\cdot\left(\nabla(-\Delta)^{-1}u\otimes\nabla(-\Delta)^{-1}u-\frac{1}{2}\left|\nabla(-\Delta)^{-1}u\right|^{2}I\right),
\end{equation*}
where $\otimes$ is a tensor product, and $I$ is a $3$-th order identity matrix, thus by using Lemma \ref{le2.7} and Lemma \ref{le2.8} with $p=\frac{6}{5-2\alpha}$, $p_{1}=2$ and $p_{2}=\frac{3}{4-\alpha}$, we obtain that
\begin{align}\label{eq5.10}
 \left\|u\nabla(-\Delta)^{-1}u\right\|_{\mathcal{L}^{1}_{t}(F\dot{B}^{\frac{3}{2}}_{\frac{6}{5-2\alpha},1})}
 &\approx\left\|\nabla(-\Delta)^{-1}u\otimes\nabla(-\Delta)^{-1}u-\frac{1}{2}\left|\nabla(-\Delta)^{-1}u\right|^{2}I\right\|_{\mathcal{L}^{1}_{t}(F\dot{B}^{\frac{5}{2}}_{\frac{6}{5-2\alpha},1})}\nonumber\\
 &\lesssim  \left\|\nabla(-\Delta)^{-1}u\right\|_{\mathcal{L}^{1}_{t}(F\dot{B}^{\frac{5}{2}}_{2,1})}\left\|\nabla(-\Delta)^{-1}u\right\|_{\mathcal{L}^{\infty}_{t}(F\dot{B}^{0}_{\frac{3}{4-\alpha},1})}\nonumber\\
 &\lesssim  \left\|\nabla(-\Delta)^{-1}u\right\|_{\mathcal{L}^{1}_{t}(F\dot{B}^{\frac{5}{2}}_{2,1})}\left\|\nabla(-\Delta)^{-1}u\right\|_{\mathcal{L}^{\infty}_{t}(F\dot{B}^{\frac{5}{2}-\alpha}_{2,1})}\nonumber\\
  &\lesssim  \left\|u\right\|_{\mathcal{L}^{1}_{t}(F\dot{B}^{\frac{3}{2}}_{2,1})}\left\|u\right\|_{\mathcal{L}^{\infty}_{t}(F\dot{B}^{\frac{3}{2}-\alpha}_{2,1})},
\end{align}
where we have used the Sobolev embedding $F\dot{B}^{\frac{5}{2}-\alpha}_{2,1}(\mathbb{R}^{3})\hookrightarrow F\dot{B}^{0}_{\frac{3}{4-\alpha},1}(\mathbb{R}^{3})$ in Lemma \ref{le2.7}. Taking the above estimates \eqref{eq5.9} and \eqref{eq5.10} into \eqref{eq5.8}, we obtain
\begin{align}\label{eq5.11}
   \left\|\mathcal{B}_{2}(u,u)\right\|_{\mathcal{L}^{\rho}_{t}(F\dot{B}^{4-\alpha-\frac{3}{p(\cdot)}+\frac{\alpha}{\rho}}_{p(\cdot),1})}&\lesssim \left\|u\right\|_{\mathcal{L}^{1}_{t}(F\dot{B}^{\frac{3}{2}}_{2,1})}\left\|u\right\|_{\mathcal{L}^{\infty}_{t}(F\dot{B}^{\frac{3}{2}-\alpha}_{2,1})}.
\end{align}

\textbf{Step 2}.  Applying Theorem \ref{th1.1} and Lemma \ref{le2.9} with $s_{1}=\frac{3}{2}$,  $s_{2}=\frac{5}{2}-\alpha$ and $p_{1}=p_{2}=2$ (the indices of the time variable obey the law of H\"{o}lder's inequality), we obtain
\begin{align}\label{eq5.12}
 \left\|\mathcal{B}_{2}(u,u)\right\|_{\mathcal{L}^{1}_{t}(F\dot{B}^{\frac{3}{2}}_{2,1})}+\left\|\mathcal{B}_{2}(u,u)\right\|_{\mathcal{L}^{\infty}_{t}(F\dot{B}^{\frac{3}{2}-\alpha}_{2,1})}&\leq \left\|\nabla\cdot(u\nabla(-\Delta)^{-1}u)\right\|_{\mathcal{L}^{1}_{t}(F\dot{B}^{\frac{3}{2}-\alpha}_{2,1})}\nonumber\\
 &\leq \left\|u\nabla(-\Delta)^{-1}u\right\|_{\mathcal{L}^{1}_{t}(F\dot{B}^{\frac{5}{2}-\alpha}_{2,1})}\nonumber\\
    &\leq C \left\|u\right\|_{\mathcal{L}^{1}_{t}(F\dot{B}^{\frac{3}{2}}_{2,1})}\left\|\nabla(-\Delta)^{-1}u\right\|_{\mathcal{L}^{\infty}_{t}(F\dot{B}^{\frac{5}{2}-\alpha}_{2,1})}
    \nonumber\\
    &\leq C \left\|u\right\|_{\mathcal{L}^{1}_{t}(F\dot{B}^{\frac{3}{2}}_{2,1})}\left\|u\right\|_{\mathcal{L}^{\infty}_{t}(F\dot{B}^{\frac{3}{2}-\alpha}_{2,1})}.
\end{align}
Putting the above estimates \eqref{eq5.11} and \eqref{eq5.12} together, we get \eqref{eq5.10}.  We complete the proof of Lemma \ref{le5.3}.
\end{proof}
\smallbreak

Based on the desired linear and nonlinear estimates obtained in Lemmas \ref{le5.1}--\ref{le5.3}, we know that there exist two positive
constants $C_{3}$ and $C_{4}$ such that
\begin{align*}
  \|u\|_{\mathcal{X}_{t}}\leq
 C_{3}\big(\|u_{0}\|_{F\dot{B}^{3-\alpha-\frac{3}{p(\cdot)}}_{p(\cdot),1}}+\|f\|_{\mathcal{L}^{1}_{t}(F\dot{B}^{3-\alpha-\frac{3}{p(\cdot)}}_{p(\cdot),1})}
 +\|f\|_{\mathcal{L}^{1}_{t}(F\dot{B}^{\frac{3}{2}-\alpha}_{2,1})}\big)+C_{4}\|u\|_{\mathcal{X}_{t}}^2.
\end{align*}
Thus if the initial data $v_{0}$ and $f$ satisfy the condition
$$
\|u_{0}\|_{F\dot{B}^{3-\alpha-\frac{3}{p(\cdot)}}_{p(\cdot),1}}+\|f\|_{\mathcal{L}^{1}_{t}(F\dot{B}^{3-\alpha-\frac{3}{p(\cdot)}}_{p(\cdot),1})}
+\|f\|_{\mathcal{L}^{1}_{t}(F\dot{B}^{\frac{3}{2}-\alpha}_{2,1})}\leq \frac{1}{4C_{3}C_{4}},
$$
we can apply Proposition \ref{pro2.12} to get global existence of solution $u\in \mathcal{Y}_{t}$ to the system \eqref{eq3.4}. We complete the proof of Theorem \ref{th3.2}.
\bigskip

\noindent \textbf{Acknowledgements.}
The authors declared that they have no conflict of interest. G. Vergara-Hermosilla thanks to Diego Chamorro for their helpful comments and advises.
J. Zhao was partially supported by the National Natural Science Foundation of China (no. 12361034) and
 the Natural Science Foundation of Shaanxi Province (no. 2022JM-034).



\begin{thebibliography}{100}
\setlength{\itemsep}{-2mm}

\bibitem{AH10}  A. Almeida, P. H\"{a}st\"{o}, Besov spaces with variable smoothness and integrability, J. Funct. Anal. 258 (2010) 1628--1655.

\bibitem{BCD11} H. Bahouri, J.-Y. Chemin, R. Danchin, \textit{Fourier Analysis and Nonlinear Partial Differential
Equations}, Grundlehren der mathematischen Wissenschaften, vol. 343.
Springer, Berlin, 2011.


\bibitem{B96} O. A. Barraza,
Self-similar solutions in weak $L^p$ spaces of the Navier--Stokes
equations, Rev. Mat. Iberoamericana 12 (1996) 411--439.

\bibitem{B99} O. A. Barraza,
Regularity and stability for the solutions of the Navier--Stokes
equations in Lorentz spaces, Nonlinear Anal.  35
(1999) 747--764.


\bibitem{BCGK04} P. Biler, M. Cannone, I.A. Guerra, G. Karch, Global regular and singular solutions for a model
of gravitating particles, Math. Ann. 330 (2004) 693--708.


\bibitem{BK10} P. Biler, G. Karch, Blowup of solutions to generalized Keller--Segel model, J. Evol. Equ. 10 (2010) 247--262.


\bibitem{BW09} P. Biler, G. Wu,
Two-dimensional chemotaxis models with fractional diffusion, Math.
Methods Appl. Sci. 32 (2009) 112--126.

\bibitem{BC10} N. Bournaveas, V. Calvez, The one-dimensional Keller--Segel model with fractional diffusion of cells,  Nonlinearity 23(4) (2010)  923--935.


\bibitem{BP08} J. Bourgain, N. Pavlovi\'{c},  Ill-posedness
of the Navier--Stokes equations in a critical space in 3D, J.
Funct. Anal. 255 (2008) 2233--2247.

\bibitem{C97} M. Cannone, A generalization of a theorem by Kato on Navier--Stokes equations, Rev. Mat.
Iberoamericana 13 (1997) 515--541.


\bibitem{CV24} D. Chamorro, G. Vergara-Hermosilla, Lebesgue spaces with variable exponent: some applications to the Navier--Stokes equations, Positivity 28(2) (2024)  p. 24.

\bibitem{C92}  J.-Y. Chemin. Remarques sur l'existence globale pour le syst\`{e}me de Navier--Stokes incompressible, SIAM J. Math. Anal. 23(1) (1992) 20--28.



\bibitem{CLW19} H. Chen, W. Lv, S. Wu, Existence for a class of chemotaxis model with fractional diffusion in Besov spaces (in Chinese), Sci. Sin. Math.  49(12) (2019) 1--17.

\bibitem{CPZ04}  L. Corrias, B. Perthame, H. Zaag, Global solutions of some chemotaxis and angiogenesis system in high space dimensions, Milan J. Math. 72 (2004)
1--28.

\bibitem{C03} D.V. Cruz-Uribe, The Hardy--Littlewood maximal operator on variable-$L^{p}$ spaces, in: Seminar of Mathematical Analysis (Malaga/Seville, 2002/2003),
147-156, Colecc. Abierta, 64, Univ. Sevilla Secr. Publ. Seville, 2003.

\bibitem{CF13} D.V. Cruz-Uribe, A. Fiorenza, \textit{Variable Lebesgue Spaces: Foundations and Harmonic Analysis}, Birkh\"{a}user/Springer, Heidelberg, 2013.

\bibitem{DL17} C. Deng, C. Liu,  Largest well-posed spaces for the general diffusion
system with nonlocal interactions, J. Funct. Anal. 272 (2017) 4030--4062.

\bibitem{D04} L. Diening, Maximal function on generalized Lebesgue spaces $L^{p(\cdot)}(\mathbb{R}^{n})$, Math. Inequal. Appl. 7 (2004) 245--253.

\bibitem{DHHR11}  L. Diening, P. Harjulehto, P. H\"{a}st\"{o}, M. Ruzicka, \textit{Lebesgue and Sobolev Spaces with
Variable Exponents}, Lecture Notes in Mathematics 2017, Springer, Heidelberg, 2011.

\bibitem{E06} C. Escudero, The fractional Keller--Segel model, Nonlinearity 19 (2006) 2909--2918.





\bibitem{FK64} H. Fujita, T. Kato,
On the Navier--Stokes initial value problem I, Arch. Rational Mech.
Anal. 16 (1964) 269--315.


\bibitem{G86}  Y. Giga, Solutions for semilinear parabolic equations in $L^p$ and regularity
of weak solutions of the Navier--Stokes system, J. Differential
Equations  62 (1986) 186--212.



\bibitem{HW13} C. Huang, B. Wang, Analyticity for the (generalized) Navier--Stokes equations with rough initail data, arXiv.13102.2141v2.

\bibitem{H03} D. Horstmann, From 1970 until present: the Keller--Segel model in chemotaxis and its consequences I, Jahresber. DMV 105 (2003) 103--165.

\bibitem{H04} D. Horstmann, From 1970 until present: the Keller--Segel model in chemotaxis and its consequences II, Jahresber. DMV 106 (2004) 51--69.


\bibitem{I11} T. Iwabuchi, Global well-posedness for Keller--Segel system in Besov type spaces, J. Math. Anal. Appl. 379 (2011) 930--948.

\bibitem{IN13} T. Iwabuchi, M. Nakamura, Small solutions for nonlinear heat equations, the
Navier--Stokes equation and the Keller--Segel system in Besov and Triebel--Lizorkin
spaces, Adv. Differential Equations 18  (2013)  687--736.

\bibitem{IT14} T. Iwabuchi, R. Takada, Global well-posedness and ill-posedness for the Navier--Stokes equations with the Coriolis force in function spaces of Besov
type, J. Funct. Anal. 267(5) (2014) 1321--1337.


\bibitem{K99} G. Karch, Scaling in nonlinear parabolic equations, J. Math. Anal. Appl. 234
(1999) 534--558.

\bibitem{K84} T. Kato, Strong $L^{p}$ solutions of the Navier--Stokes equations in $\mathbb{R}^{m}$ with applications to weak
solutions, Math. Z. 187 (1984) 471--480.

\bibitem{KS70} E.F. Keller, L.A. Segel, Initiation of slime mold aggregation viewed as an instability, J. Theoret. Biol. 26 (1970) 399--415.

\bibitem{K09} H. Kempka, $2$-Microlocal Besov and Triebel--Lizorkin spaces of variable integrability, Rev. Mat. Complut. 22(1)
(2009) 227--251.

\bibitem{KT01} H. Koch,  D. Tataru, Well-posedness for the Navier--Stokes equations, Adv. Math. 157
(2001)  22--35.

\bibitem{KY11} P. Konieczny, T. Yoneda, On dispersive effect of the Coriolis force for the stationary Navier--Stokes equations, J. Differential Equations 250(10) (2011)
3859--3873.

\bibitem{KR91} O. Kov\'{a}\v{c}ik, J. R\'{a}kosn\'{i}k, On spaces $L^{p(x)}$ and $W^{k, p(x)}$, Czechoslovak Math. J. 41(116) (1991) 592--618.

\bibitem{KS08} H. Kozono, Y. Sugiyama, Local existence and finite time blow-up of solutions in the 2-D Keller--Segel system, J. Evol. Equ. 8 (2008) 353--378.


\bibitem{L02} P.-G. Lemari\'{e}-Rieusset, \textit{Recent Developments in the Navier--Stokes Problem}, Research
Notes in Mathematics, Chapman \& Hall/CRC, 2002.


\bibitem{L16} P. G. Lemari\'{e}-Rieusset, \textit{The Navier--Stokes Problem in the 21st Century}, A Chapman \& Hall Book, CRC Press, 2016.


\bibitem{LW23}  Y. Li, W. Wang, Finite-time blow-up and boundedness in a 2D Keller--Segel system with rotation, Nonlinearity
36(1) (2023) 287--318.


\bibitem{LYZ23}  J. Li, Y. Yu, W. Zhu, Ill-posedness issue on a multidimensional chemotaxis equations in the critical Besov spaces,  J. Geom. Anal. 33:84 (2023) 1--22.

\bibitem{LZ10} P. Li, Z. Zhai, Well-posedness and regularity of generalized Navier--Stokes equations in some critical $Q$-spaces, J. Funct. Anal. 259 (2010) 2457--2519.

\bibitem{L69} J. L. Lions, \textit{Quelques m\'{e}thodes de r\'{e}solution des probl\`{e}mes aux limites non lin\'{e}aires}, Dunod,
Gauthier-Villars, Paris, 1969.

\bibitem{MGS13} J.M. Mercado, E.P. Guido, E.P. S\'{a}nchez-Sesma, A.J. \'{I}niguez, A. Gonz\'{a}lez, Analysis of
the Blasius's formula and the Navier--Stokes fractional equation, in: Fluid Dynamics in Physics,
Engineering and Environmental Applications, Environmental Science and Engineering, Springer,
Berlin, Heidelberg, pp. 475--480, 2013.

\bibitem{MYZ08} C. Miao, B. Yuan, B. Zhang, Well-posedness of the Cauchy problem for the fractional
power dissipative equations, Nonlinear Anal. 68 (2008) 461--484.

\bibitem{N50}  H. Nakano, \textit{Modulared Semi-Ordered Linear Spaces}, Maruzen Co. Ltd., Tokyo, 1950.

\bibitem{N51} H. Nakano, \textit{Topology of Linear Topological Spaces}, Maruzen Co. Ltd., Tokyo, 1951.

\bibitem{NY20}  Y. Nie, J. Yuan, Well-posedness and ill-posedness of a multidimensional chemotaxis system in the
critical Besov spaces, Nonlinear Anal. 196 (2020) 111782.

 \bibitem{NY22} Y. Nie, J. Yuan, Ill-posedness issue for a multidimensional hyperbolic-parabolic model of chemotaxis
in critical Besov spaces $\dot{B}^{\frac{3}{2}}_{2d,1}\times (\dot{B}^{-\frac{1}{2}}_{2d,1})^{d}$, J. Math. Anal. Appl. 505 (2) (2022) 125539.


\bibitem{OS08} T. Ogawa, S. Shimizu, The drift-diffusion system in two-dimensional critical Hardy space,
J. Funct. Anal. 255 (2008) 1107--1138.

\bibitem{OS10} T. Ogawa, S. Shimizu, End-point maximal regularity and wellposedness of
the two dimensional Keller--Segel system in a critical Besov space,
Math. Z.  264 (2010) 601--628.

\bibitem{O31} W. Orlicz, \"{U}ber konjugierte Exponentenfolgen, Studia Math. 3 (1931) 200--212.

\bibitem{OY01}  K. Osaki, A. Yagi,  Finite dimensional attractors for one-dimensional Keller--Segel equations, Funkcialaj Ekvacioj.
44 (2001) 441--469.

\bibitem{P96} F. Planchon, Global strong solutions in Sobolev or Lebesgue spaces to the incompressible
Navier--Stokes equations in $\mathbb{R}^{3}$, Ann. Inst. H. Poincar\'{e} 13 (1996) 319--336.



\bibitem{R18} S. Ru, Global well-posedness of the incompressible Navier--Stokes
equations in function spaces with variable exponents (in Chinese), Sci. Sin. Math. 48 (2018) 1427--1442.

\bibitem{RA19}  S. Ru, M.Z. Abidin, Global well-posedness of the incompressible fractional
Navier--Stokes equations in Fourier--Besov spaces with
variable exponents, Comp. Math. Appl. 77 (2019) 1082--1090.

\bibitem{SYK15} Y. Sugiyama, M. Yamamoto, K. Kato, Local and global solvability and blow up for the drift-diffusion equation with the fractional dissipation inthe critical space, J. Differential Equations 258 (2015) 2983--3010.



\bibitem{V241} G. Vergara-Hermosilla, Remarks on variable Lebesgue spaces and fractional Navier--Stokes equations, arXiv:2402.07508v1.


\bibitem{VZ241} G. Vergara-Hermosilla, J. Zhao, On the Cauchy problem for the fractional Keller--Segel system in variable Lebesgue spaces, arXiv:2405.01209v1.

\bibitem{VZ242} G. Vergara-Hermosilla, J. Zhao, Global well-posedness of the Navier--Stokes equations and the Keller--Segel system in variable Fourier--Besov spaces, 2410.05293v1.

\bibitem{V09} J. Vyb\'{i}ral, Sobolev and Jawerth embeddings for spaces with variable smoothness and integrability, Ann. Acad. Sci.
Fenn. Math. 34 (2) (2009) 529--544.

\bibitem{W15} B. Wang,  Ill-posedness for the Navier--Stokes equations in critical Besov spaces $\dot{B}_{\infty,q}^{-1}$,  Adv. Math.  268  (2015) 350--372.



\bibitem{W03} J. Wu, Generalized MHD equations, J. Differential Equations 195 (2003) 284--312.

\bibitem{W05} J. Wu, Lower bounds for an integral involving fractional Laplacians and the generalized Navier--Stokes equations
in Besov spaces, Commun. Math. Phys. 263 (2005) 803--831.



\bibitem{WZ11} G. Wu, X. Zheng, On the well-posedness for Keller--Segel system with fractional diffusion,
Math. Methods Appl. Sci. 34(14) (2011) 1739--1750.

\bibitem{XF22}W. Xiao, X. Fei, Ill-posedness of a multidimensional chemotaxis system in the critical Besov spaces,
J. Math. Anal. Appl. 514 (2022)  126302.

\bibitem{YKS14} M. Yamamoto, K. Kato, Y. Sugiyama,  Existence and analyticity of solutions to
the drift-diffusion equation with critical dissipation. Hiroshima Math. J. 44  (2014) 275--313.

\bibitem{Y00} M. Yamazaki,
The Navier--Stokes equations in the weak-$L^n$ spaces with
time-dependent external force, Math. Ann. 317 (2000)
635--675.

\bibitem{Y10} T. Yoneda, Ill-posedness of the 3D Navier--Stokes equations in a generalized Besov space near $BMO^{-1}$, J. Funct. Anal. 258(10) (2010) 3376--3387.

\bibitem{YZ14} X. Yu, Z. Zhai, Well-posedness for fractional Navier--Stokes equations in the largest critical spaces $\dot{B}^{-(2\beta-1)}_{\infty,\infty}(\mathbb{R}^{n})$, Math. Methods Appl. Sci. 35 (2014) 676--683.


\bibitem{Z10} Z. Zhai, Global well-posedness for nonlocal fractional
Keller--Segel systems in critical Besov spaces, Nonlinear Anal. 72
(2010) 3173--3189.

\bibitem{Z12} X. Zhang,  Stochastic Lagrangian particle approach to fractal Navier--Stokes equations, Comm.
Math. Phys. 311(1) (2012) 133--155.

\bibitem{Z18}  J. Zhao, Well-posedness and Gevrey analyticity of the generalized
Keller--Segel system in critical Besov spaces, Annali di Matematica 197 (2018)  521--548.

\bibitem{Z21} J. Zhao, Global existence of large solutions for the generalized Poisson--Nernst--Planck equations,  J. Math. Anal. Appl.  498 (2021)  124943.

\bibitem{Z24} J. Zhao, Global existence of large solutions for the parabolic-elliptic Keller--Segel system in Besov type spaces, Applied Mathematics Letters 149 (2024) 108899.

\bibitem{ZZ18} W. Zhu, J. Zhao, Existence and regularizing rate estimates of solutions to the
3-D generalized micropolar system in Fourier--Besov spaces,  Math. Methods Appl. Sci. 41 (2018) 1703--1722.

\bibitem{ZZ19} W. Zhu, J. Zhao, The optimal temporal decay estimates for the micropolar
fluid system in negative Fourier--Besov spaces,  J. Math. Anal. Appl. 475 (2019) 154--172.
\end{thebibliography}
\end{document}